\documentclass[reqno,11pt]{amsart}
\usepackage{amsfonts}
\usepackage{amsmath}
\usepackage{amsthm, amscd}
\usepackage{mathtools}
\usepackage{amssymb}
\usepackage{mathrsfs}
\usepackage{graphicx}
\usepackage{mathabx}
\usepackage{bbm}

\usepackage[alphabetic]{amsrefs}
\usepackage{etex}
\usepackage[hidelinks]{hyperref}

\usepackage{xypic}
\usepackage{geometry}
\usepackage{fancyhdr}
\usepackage{color}
\usepackage{overpic}
\usepackage{tikz-cd}

\allowdisplaybreaks

\theoremstyle{plain}\newtheorem{thm}{Theorem}[section]
\theoremstyle{plain}\newtheorem{cor}[thm]{Corollary}
\theoremstyle{plain}\newtheorem{prop}[thm]{Proposition}
\theoremstyle{plain}\newtheorem{lem}[thm]{Lemma}
\theoremstyle{plain}\newtheorem{conj}[thm]{Conjecture}
\theoremstyle{plain}
\theoremstyle{plain}
\theoremstyle{plain}
\theoremstyle{plain}
\theoremstyle{plain}\newtheorem{property}[thm]{Property}

\theoremstyle{definition}
\newtheorem{defn}[thm]{Definition}

\newtheorem{exmp}[thm]{Example}

\newtheorem{notn}[thm]{Notation}

\theoremstyle{remark}
\newtheorem{rem}[thm]{Remark}

\makeatletter
\newtheorem*{rep@theorem}{\rep@title}
\newcommand{\newreptheorem}[2]{%
	\newenvironment{rep#1}[1]{%
		\def\rep@title{#2 \ref{##1}}%
		\begin{rep@theorem}}%
		{\end{rep@theorem}}}
\makeatother
\newreptheorem{theorem}{Theorem}

\numberwithin{equation}{section}

\DeclareMathOperator{\rank}{rank}
\DeclareMathOperator{\II}{I}

\DeclareMathOperator{\SHI}{SHI}
\DeclareMathOperator{\AHI}{AHI}
\DeclareMathOperator{\MHI}{MHI}

\DeclareMathOperator{\APS}{APS}

\DeclareMathOperator{\CKh}{CKh}
\DeclareMathOperator{\muu}{\mu^{orb}}

\DeclareMathOperator{\id}{id}

\DeclareMathOperator{\Eig}{Eig}
\DeclareMathOperator{\KHI}{KHI}

\DeclareMathOperator{\SiHI}{\Sigma HI}

\newcommand{\bC}{\mathbb{C}}

\newcommand{\bR}{\mathbb{R}}

\newcommand{\bZ}{\mathbb{Z}}

\newcommand{\bfv}{\mathbf{v}}
\newcommand{\bfw}{\mathbf{w}}
\newcommand{\clP}{\mathcal{P}}

\newcommand{\al}{\alpha}

\newcommand{\be}{\beta}
\newcommand{\ga}{\gamma}

\newcommand{\bpf}{\begin{proof}}
\newcommand{\epf}{\end{proof}}
\newcommand{\bthm}{\begin{thm}}
\newcommand{\ethm}{\end{thm}}
\newcommand{\bprop}{\begin{prop}}
\newcommand{\eprop}{\end{prop}}
\newcommand{\bcor}{\begin{cor}}
\newcommand{\ecor}{\end{cor}}
\newcommand{\blem}{\begin{lem}}
\newcommand{\elem}{\end{lem}}
\newcommand{\bdefn}{\begin{defn}}
\newcommand{\edefn}{\end{defn}}
\newcommand{\bexmp}{\begin{exmp}}
\newcommand{\eexmp}{\end{exmp}}
\newcommand{\brem}{\begin{rem}}
\newcommand{\erem}{\end{rem}}

\newcommand{\bdia}{\begin{displaymath}\xymatrix}
\newcommand{\edia}{\end{displaymath}}
\newcommand{\beq}{\begin{equation*}\begin{aligned}}
\newcommand{\eeq}{\end{aligned}\end{equation*}}

\newcommand{\mfo}{\mathfrak{o}}

\author{Zhenkun Li}
\address{Department of Mathematics, Stanford University, California 94305, USA}
\email{zhenkun@stanford.edu}
\author{Yi Xie}
\address{Beijing International Center for Mathematical Research, Peking University, Beijing 100871, China}
\email{yixie@pku.edu.cn}
\author{Boyu Zhang}
\address{Mathematics Department, University of Maryland at College Park, Maryland 20742, USA}
\email{bzh@umd.edu}

\title{Instanton homology and knot detection on thickened surfaces}

\begin{document}

\begin{abstract}
	Suppose $\Sigma$ is a compact oriented surface (possibly with boundary) that has genus zero, and $L$ is a link in the interior of $(-1,1)\times \Sigma$. We prove that the Asaeda--Przytycki--Sikora (APS) homology of $L$ has rank $2$ if and only if $L$ is isotopic to an embedded knot in $\{0\}\times \Sigma$. As a consequence, the APS homology detects the unknot in $(-1,1)\times \Sigma$. This is the first detection result for generalized Khovanov homology that is valid on an infinite family of manifolds, and it partially solves a conjecture in  \cite{xie2020instantons}.  Our proof is different from the previous detection results obtained by instanton homology because in this case, the second page of Kronheimer--Mrowka's spectral sequence is not isomorphic to the APS homology. We also characterize all links in  product manifolds that have minimal sutured instanton homology, which may be of independent interest.
\end{abstract}

\maketitle

\section{Introduction}

 Khovanov homology \cite{Kh-Jones} is a link invariant that assigns a homology group to every link in $S^3$. 
 A generalization of Khovanov homology was introduced by Asaeda--Przytycki--Sikora \cite{APS} for links in the interior of $[-1,1]\times \Sigma$, where $\Sigma$ is a compact oriented\footnote{
 	When $\Sigma$ is non-orientable, there is a unique $(-1,1)$--bundle over $\Sigma$ such that the total manifold is orientable, and the APS homology can be defined for links in this bundle.  We will only consider the case when $\Sigma$ is oriented in this paper.} surface  possibly with boundary. 
 We will call this invariant the \emph{APS homology} and denote the APS homology of a link $L$ by $\APS(L)$.  APS homology recovers Khovanov homology when $\Sigma$ is a disk or sphere.
 
Kronheimer and Mrowka \cite{KM:Kh-unknot} proved that Khovanov homology distinguishes the unknot from other knots and links (see also \cite{Hedden-unknot, Grigssby-Wehrli-color}). 
A natural question is whether the analogous result holds for APS homology. More generally, the following conjecture was stated in \cite{xie2020instantons}.
\begin{conj}
	\label{conj_main}
	Suppose $\Sigma$ is a compact oriented surface and $L$ is a link in the interior of $[-1,1]\times \Sigma$. Then $\rank_{\bZ/2}\APS(L;\bZ/2)\ge 2$, and the equality holds if and only if $L$ is isotopic to a knot embedded in $\{0\}\times \Sigma$. 
\end{conj}
 It follows immediately from the definition that  for every knot $K$ embedded in $\{0\}\times \Sigma$, we have  
 $\rank_{\bZ/2}\APS(K;\bZ/2)=2;$ two knots $K_1,K_2\subset \{0\}\times \Sigma$ are isotopic if and only if their APS homology have the same gradings.  Therefore, if Conjecture \ref{conj_main} is true, then for every knot $K\subset \{0\}\times \Sigma$, the APS homology distinguishes (the isotopy class of) $K$ from other knots and links.

The unknot detection theorem of Kronheimer--Mrowka gives a positive answer to Conjecture \ref{conj_main} when $\Sigma$ is a disk or sphere.
Special cases of Conjecture \ref{conj_main} were also known when $\Sigma$ is an annulus \cite{XZ:excision} 
or a torus \cite{xie2020instantons}. 
The main result of this paper gives a positive answer to Conjecture \ref{conj_main} for all compact surfaces with genus zero, possibly with boundary.

\begin{thm}
	\label{thm_main}
		Suppose $\Sigma$ is a compact oriented surface with genus zero. Let $L$ be a link in the interior of $[-1,1]\times \Sigma$. Then $\rank_{\bZ/2}\APS(L;\bZ/2)\ge 2$, and the equality holds if and only if $L$ is isotopic to a knot embedded in $\{0\}\times \Sigma$. 
\end{thm}

Theorem \ref{thm_main} is the first detection result for generalized Khovanov homology that is valid on an infinite family of manifolds. Although our proof also uses instanton Floer homology and spectral sequences, there are multiple essential difficulties that make the proof of Theorem \ref{thm_main} conceptually different from the earlier detection results in \cite{KM:Kh-unknot, XZ:excision, xie2020instantons}. 

First of all, when $\Sigma$ has genus zero and has at least three boundary components, the second page of Kronheimer--Mrowka's spectral sequence is \emph{not} isomorphic to APS homology. We will construct another spectral sequence, which is inspired by the work of Winkeler \cite{winkeler2021khovanov}, to relate the second page of Kronheimer--Mrowka's spectral sequence to APS homology.

It is also significantly more difficult to compute the differentials on the first page of Kronheimer--Mrowka's spectral sequence in the more general setting. Although the isomorphism class of the underlying group can be computed by sutured homology (Proposition \ref{prop_SiHI_iso_SHI}), due to the lack of a strong enough naturality property,  one cannot use it to compute the differential maps. We will develop a different argument 
by establishing an invariance property of instanton homology under diffeomorphisms (Proposition \ref{prop_diff_induces_id_on_II}) and using it to obtain a partial algebraic description of the differential map (Proposition \ref{prop_description_of_T(b)}). It turns out that the partial description of the differentials given by Proposition \ref{prop_description_of_T(b)} will be sufficient for the proof of the main theorem.

More importantly, even after a relationship between instanton homology and APS homology is established, the desired detection result does not follow immediately from the available tool package of instanton homology.  This is because when the Euler characteristic of $\Sigma$ is negative, the usual argument of estimating Thurston norms using instanton homology is no longer effective in restricting the isotopy class of $L$. We resolve this difficulty by establishing the following result on sutured instanton Floer homology, which may be of independent interest.

\begin{thm}
		\label{thm_main_instanton}
Suppose $\Sigma$ is a connected compact oriented surface with non-empty boundary, and $L$ is a link in the interior of $[-1,1]\times \Sigma$. If the singular sutured instanton homology $\SHI([-1,1]\times \Sigma, \{0\}\times \Sigma, L)$ has dimension no greater than $2$, then  $L$ is isotopic to a knot  in $\{0\}\times \Sigma$.
\end{thm}

The paper is organized as follows. Section \ref{sec_preliminaries} reviews some basic constructions and terminologies that will be needed in this paper, Section \ref{sec_two_nonvanishing_grading} proves
Theorem \ref{thm_main_instanton}, Sections \ref{sec_instanton_thickend_surface} and \ref{sec_instanton_for_links_at_0} establish various properties of instanton Floer homology that will be needed in the proof of Theorem \ref{thm_main}, Section \ref{sec_proof_main} finishes the proof of the main theorem.

\vspace{\baselineskip}
{\bf Acknowledgments.} We would like to thank Ciprian Manolescu for helpful comments, and thank Lvzhou Chen for helpful discussions.

\section{Preliminaries}
\label{sec_preliminaries}

\subsection{Notation and conventions}
In this paper, all manifolds and all maps between manifolds are smooth. 
If $M$ is an oriented manifold, we use $-M$ to denote the same manifold $M$ with the opposite orientation. We use ${\rm int}(M)$ to denote the interior of a manifold $M$.

Unless otherwise specified, all submanifolds of $\bR$ are endowed with the inherited orientation from the standard orientation of $\bR$, and product manifolds are endowed with the product orientation. 

We view $S^1$ as the quotient of $[-1,1]$ by gluing $1$ and $-1$. The standard orientation on $S^1$ is defined to be the push-forward of the standard orientation on $[-1,1]$. We will use $t_*\in S^1$ to denote the image of $1$ and $-1$ under the quotient map. The point $t_*\in S^1$ will often be used as a base point on $S^1$.

We will use $A$ to denote the oriented annulus $[-1,1]\times S^1$. 

If $\Sigma$ is a compact oriented surface, we will call the manifold $(-1,1)\times \Sigma$ the \emph{thickened} $\Sigma$. 

\subsection{APS homology}
\label{subsec_APS}
This subsection briefly reviews the definition of APS homology from \cite{APS}. 
We will not discuss the gradings of APS homology here because they are not needed for Theorem \ref{thm_main}. 
For simplicity, we will focus on the case when $\Sigma$ is a compact oriented surface with genus zero. 
We also slightly modify the notation from \cite{APS} to make it more convenient for our proof of Theorem \ref{thm_main}.

Suppose $C$ is an embedded compact closed $1$--manifold in $\Sigma$. Define a free abelian group $V(C)$ as follows:
\begin{enumerate}
	\item  If $\ga$ is a contractible simple closed curve on $\Sigma$, define $V(\ga)$ to be the free group generated by $\bfv_+(\ga)$ and $\bfv_-(\ga)$, where $\bfv_+(\ga)$ and $\bfv_-(\ga)$ are formal generators associated with $\ga$. 
	\item If $\ga$ is a non-contractible simple closed curve, let $\mathfrak{o}$, $\mathfrak{o}'$ be the two orientations of $\ga$. Define $V(\ga)$ to be the free group generated by $\bfw_{\mfo}(\ga)$ and $\bfw_{\mfo'}(\ga)$, where $\bfw_{\mfo}(\ga)$ and $\bfw_{\mfo'}(\ga)$ are formal generators associated with $\ga$.
	\item  In general, suppose the connected components of $C$ are $\ga_1,\dots,\ga_k$, define 
	$$V(C)= \bigotimes_{i=1}^k V(\ga_i).$$
\end{enumerate}

Suppose a link $L$ in the interior of $[-1,1]\times \Sigma$ is given by a diagram $D$ on $\Sigma$ with $k$ crossings, and fix an ordering of the crossings.
For $v=(v_1,v_2,\dots, v_k)\in \{0,1\}^k$, resolving the crossings of $D$ by a sequence of $0$--smoothings and $1$--smoothings (see Figure \ref{fig_01smoothing}) given by $v$ turns $D$ to an embedded closed $1$--manifold in $\Sigma$. Denote the resolved diagram by $D_v$.

\begin{figure}
	\begin{tikzpicture}
		\draw[thick] (1,-1) to (-1,1); \draw[thick,dash pattern=on 1.3cm off 0.25cm] (1,1) to (-1,-1);  \node[below] at (0,-1.1) {A crossing};
		\draw[thick] (2,1)  to [out=315,in=180]  (3,0.3) to [out=0,in=225]   (4,1);
		\draw[thick] (2,-1)  to [out=45,in=180]  (3,-0.3) to [out=0,in=135]   (4,-1);  \node[below] at (3,-1.1) {0-smoothing};
		\draw[thick] (5,1)  to [out=315,in=90]  (5.7,0) to [out=270,in=45]    (5,-1);  \node[below] at (6,-1.1) {1-smoothing};
		\draw[thick] (7,1)  to [out=225,in=90]  (6.3,0) to [out=270,in=135]   (7,-1);
	\end{tikzpicture}
	\caption{Two types of smoothings}\label{fig_01smoothing}
\end{figure}
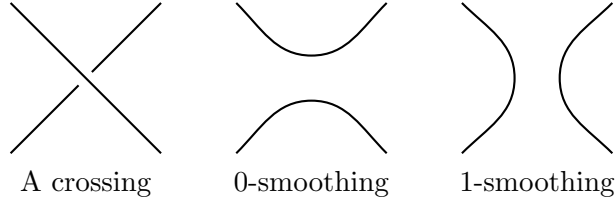

If a $0$--smoothing is changed to a $1$--smoothing at a given crossing, then the resolved diagram changes by merging two circles into one circle, or splitting one circle to two circles.
Let $u\in \{0,1\}^k$ be the label of the new resolution, we define a map from $V(D_v)$ to $V(D_u)$ as follows. 
We can write $D_v = D_v' \cup C$ and $D_u = D_u' \cup C$, where $D_v'$ and $D_u'$ are the components of $D_v$ and $D_u$ that are modified by the smoothing change, and $C$ is the complement. 
We define a map from $V(D_v')$ to $V(D_u')$; this will induce a map from $V(D_v)$ to $V(D_u)$ by taking tensor product with the identity map on $V(C)$. 

In the case that $D_v'$ contains two circles and $D_u'$ contains one circle, write $D_v'= \gamma_1\cup \gamma_2$ and $D_u'=\gamma$.
\begin{enumerate}
	\item  If $\gamma_1$ and $\gamma_2$ are both contractible circles, the map is defined by
	\begin{align*}
		\mathbf{v}_+(\gamma_1) \otimes \mathbf{v}_+(\gamma_2) &\mapsto \mathbf{v}_+(\gamma),  &\mathbf{v}_+(\gamma_1)& \otimes \mathbf{v}_-(\gamma_2) \mapsto \mathbf{v}_-(\gamma) ,\\
		\mathbf{v}_-(\gamma_1) \otimes \mathbf{v}_+(\gamma_2)&\mapsto \mathbf{v}_-(\gamma),  &\mathbf{v}_-(\gamma_1)& \otimes \mathbf{v}_-(\gamma_2) \mapsto 0.
	\end{align*}
\item If $\gamma_1$  is contractible and $\gamma_2$ is non-contractible, the map is defined by
$$
	\mathbf{v}_+(\gamma_1)\otimes \mathbf{w}_{\mfo}(\gamma_2) \mapsto \mathbf{w}_{\mfo}(\gamma_2),\quad\quad \mathbf{v}_-(\gamma_1)\otimes \mathbf{w}_\mfo(\gamma_2) \mapsto 0
$$
for every orientation $\mfo$ of $\ga_2$. 
\item
If both $\gamma_1$ and $\gamma_2$ are non-contractible and $\ga$ is contractible, then $\gamma_1$ and $\gamma_2$ must be parallel to each other, and we can identify the orientations of $\ga_1$ with the orientations of $\ga_2$. Let $\mfo$ and $\mfo'$ be the orientations of $\ga_1$ and use the same notation to denote the corresponding orientations of $\ga_2$. The merge map is defined by
\begin{align*}
	\mathbf{w}_{\mfo}(\gamma_1)\otimes \mathbf{w}_{\mfo}(\gamma_2) &\mapsto 0,   &\mathbf{w}_{\mfo'}(\gamma_1)& \otimes \mathbf{w}_{\mfo'}(\gamma_2) \mapsto 0,\\
	\mathbf{w}_{\mfo}(\gamma_1)\otimes \mathbf{w}_{\mfo'}(\gamma_2) &\mapsto \mathbf{v}_-(\gamma), &\mathbf{w}_{\mfo'}(\gamma_1)&\otimes \mathbf{w}_{\mfo}(\gamma_2) \mapsto \mathbf{v}_-(\gamma).
\end{align*}
\item If $\ga_1,\ga_2,\ga$ are all non-contractible, define the merge map to be zero.
\end{enumerate}

In the case that $D_v'$ contains one circle and $D_u'$ contains two circles, write $D_v'= \gamma$ and $D_u'=\gamma_1\cup \gamma_2$.
\begin{enumerate}
	\item  If $\gamma_1$ and $\gamma_2$ are both contractible circles, then $\ga$ is also contractible, and the map is defined by
	$$
	\mathbf{v}_+(\gamma) \mapsto \mathbf{v}_+(\gamma_1)\otimes \mathbf{v}_-(\gamma_2)+ \mathbf{v}_-(\gamma_1)\otimes \mathbf{v}_+(\gamma_2), \quad \mathbf{v}_-(\gamma) \mapsto \mathbf{v}_-(\gamma_1)\otimes \mathbf{v}_-(\gamma_2).
	$$
	\item If $\gamma_1$  is contractible and $\gamma_2$ is non-contractible, then $\ga_2$ is parallel to $\ga$. Identify the orientations of $\ga$ with the orientations of $\ga_2$. The split map is defined by
	$$
	\mathbf{w}_{\mfo}(\gamma) \mapsto \mathbf{v}_-(\gamma_1)\otimes \mathbf{w}_{\mfo}(\gamma_2)
	$$
	for every orientation $\mfo$ of $\ga$ and $\ga_2$.
	\item 
	If both $\ga_1$ and $\ga_2$ are non-contractible and $\ga$ is contractible, then $\gamma_1$ and $\gamma_2$ must be parallel to each other, and we can identify the orientations of $\ga_1$ with the orientations of $\ga_2$. Let $\mfo$ and $\mfo'$ be the orientations of $\ga_1$ and use the same notation to denote the corresponding orientations of $\ga_2$. The split map is given by 
	$$
	\mathbf{v}_+(\gamma) \mapsto \mathbf{w}_{\mfo}(\gamma_1)\otimes \mathbf{w}_{\mfo'}(\gamma_2) + \mathbf{w}_{\mfo'}(\gamma_1)\otimes \mathbf{w}_{\mfo}(\gamma_2) , \quad \mathbf{v}_-(\gamma) \mapsto 0.
	$$
	\item If all of $\ga_1,\ga_2,\ga$ are non-contractible, define the splitting map to be zero.
\end{enumerate}

The above construction defines a map $d_{vu}: V(D_v) \to V(D_u)$ whenever $u$ is obtained from $v$ by changing one coordinate from 0 to 1. Let $e_i$ be the $i$--th standard basis vector of $\mathbb{Z}^k$.
Define
\begin{equation*}
	\CKh_\Sigma(L)=\bigoplus_{v\in \{0,1\}^k} V(D_v),
\end{equation*}
and define an endomorphisms on $\CKh_\Sigma(L)$ by
\begin{equation*}
	\mathcal{D}_\Sigma= \sum_i \sum_{u-v=e_i}  (-1)^{\sum_{i<j\le c}v_j}  d_{vu}.
\end{equation*}
\begin{thm}[\cite{APS}]
	The map $\mathcal{D}_\Sigma$ satisfies $\mathcal{D}_\Sigma^2=0$. The homology of $(\CKh_\Sigma(L),\mathcal{D}_\Sigma)$ does not depend on the diagram $D$ or the order of the crossings.
\end{thm}

The homology of $(\CKh_\Sigma(L),\mathcal{D}_\Sigma)$ is the APS homology of $L$. If $R$ is a commutative ring, the APS homology of $L$ with coefficient $R$ is defined to be the homology of $\CKh_\Sigma(L)\otimes R$ with respect to the differential $\mathcal{D}_\Sigma\otimes \id_R$, and we denote the homology by $\APS(L;R)$.

\subsection{Singular instanton homology}
We give a brief review of the theory of singular instanton Floer homology developed in \cite{KM:Kh-unknot,KM:YAFT}.
Let $Y$ be a closed oriented $3$--manifold, $L$ be a link in $Y$, and $\omega\subset Y$ be an embedded compact $1$--manifold such that $\omega \cap L = \partial \omega$. A closed, oriented, connected, embedded surface $\Sigma\subset Y$ is called \emph{non-integral} (with respect to the triple $(Y,L,\omega)$) if one of the following conditions hold:
\begin{enumerate}
	\item The algebraic intersection number of $\Sigma$ and $L$ is odd.
	\item $\Sigma$ is disjoint from $L$, and the intersection number of $\Sigma$ and $\omega$ is odd.
\end{enumerate}
Note that when $\Sigma$ is disjoint from $L$, the mod $2$ intersection number of $\Sigma$ and $\omega$ is well-defined, so condition (2) above is well-defined. The triple $(Y,L,\omega)$ is called \emph{admissible} if there exists a non-integral surface in every connected component of $Y$.

Let $R$ be a ring. 
\emph{Singular instanton homology} assigns an $R$--module $\II(Y,L,\omega,a;R)$ to every admissible triple $(Y,L,\omega)$ and a choice of ``auxiliary data'' $a$ on $(Y,L,\omega)$ (see \cite{KM:Kh-unknot}). 
The homology $\II(Y,L,\omega, a;R)$ is endowed with a relative $\bZ/4$ grading. 
For different choices of auxiliary data $a_1$ and $a_2$, the homology $\II(Y,L,\omega,a_1;R)$ and $\II(Y,L,\omega,a_2;R)$ are isomorphic as $\bZ/4$--relatively graded $R$--modules. The isomorphism is canonical up to an overall sign. We will often omit the auxiliary data $a$ from the notation and write the singular instanton homology group as $\II(Y,L,\omega;R)$. We always assume that some auxiliary data is chosen when defining singular instanton homology. 

If $\omega_1$ and $\omega_2$ have the same fundamental class in $H_1(Y,L;\bZ/2)$, then $\II(Y,L,\omega_1;R)\cong \II(Y,L,\omega_2;R)$ as $\bZ/4$ relatively graded modules.

Suppose $(Y_1,L_1,\omega_1)$ and $(Y_2,L_2,\omega_2)$ are two admissible triples. A \emph{cobordism} from $(Y_1,L_1,\omega_1)$ to $(Y_2,L_2,\omega_2)$ is a quadruple $(W,S,\Omega,\tau)$, such that 
\begin{enumerate}
	\item $W$ is an oriented $4$--manifold with boundary.
	\item $(S,\partial S)\subset (W,\partial W)$ is a properly embedded surface such that the interior of $S$ is disjoint from $\partial W$.
	\item $(\Omega, \partial \Omega)\subset (W,S\cup \partial W)$ is a tamely embedded surface with corner such that the interior of $\Omega$ is disjoint from $S \cup \partial W$.
	\item $\tau$ is an orientation-preserving diffeomorphism from $\partial W$ to $(-Y_1)\sqcup Y_2$. 
	\item $\tau(\partial S)= L_1\cup L_2$.
	\item $\tau(\Omega\cap \partial W) = \omega_1\cup \omega_2$.
\end{enumerate}

A cobordism $(W,S,\Omega,\tau)$ from $(Y_1,L_1,\omega_1)$ to $(Y_2,L_2,\omega_2)$ induces a homomorphism 
$$
\II(W,S,\Omega,\tau): \II(Y_1,L_1,\omega_1;R) \to \II(Y_2,L_2,\omega_2;R).
$$
The map $\II(W,S,\Omega,\tau)$ is only well defined up to an overall sign. Therefore, all identities and commutative diagrams involving cobordism maps only hold up to signs unless specific choices of signs are made.

If $(W,S,\Omega,\tau)$ and $(W',S',\Omega',\tau')$ are two cobordisms from $(Y_1,L_1,\omega_1)$ to $(Y_2,L_2,\omega_2)$ such that there exists a diffeomorphism $\eta$ from $(W,S,\Omega)$ to $(W',S',\Omega')$ with $\tau'\circ\eta = \tau$, then $\II(W,S,\Omega,\tau) = \pm \II(W',S',\Omega',\tau')$.

From now on, we will omit $\tau$ from the notation when it is clear from context.

\begin{rem}
	It is also possible to define a cobordism map when $S$ has normal self-intersection points and when $\Omega$ has a non-trivial intersection with $S$ in the interior. Since they are not needed in this paper, we will not discuss these constructions here.
\end{rem}

The cobordism maps satisfy a functoriality property in the following sense.  
\begin{enumerate}
	\item If $(W,S,\Omega)$ is a product cobordism, then $\II(W,S,\Omega) = \pm \id$.
	\item 
	Suppose $(W_1,S_1,\Omega_1)$ is a cobordism from $(Y_1,L_1,\omega_1)$ to $(Y_2,L_2,\omega_2)$ and $(W_2,S_2,\Omega_2)$ is a cobordism from $(Y_2,L_2,\omega_2)$ to $(Y_3,L_3,\omega_3)$. Then the gluing of $(W_1,S_1,\Omega_1)$ and $(W_2,S_2,\Omega_2)$ along $Y_2$ defines a cobordism $(\hat{W},\hat{S},\hat{\Omega})$ from $(Y_1,L_1,\omega_1)$ to $(Y_3,L_3,\omega_3)$. If $Y_2$ has $n$ components and $W_1$, $W_2$ are connected, then 
	$$
     \II(\hat{W},\hat{S},\hat{\Omega}) = \pm 2^{n-1} 	\II(W_2,S_2,\Omega_2)\circ \II(W_1,S_1,\Omega_1) .
	$$
\end{enumerate}

\begin{exmp}
	\label{exmp_diff_induce_iso}
If $\varphi$ is a diffeomorphism from $(Y_1,L_1,\omega_1)$ to $(Y_2,L_2,\omega_2)$ that preserves the orientations of $Y_i$, then $\varphi$ defines an isomorphism from $\II(Y_1,L_1,\omega_1)$ to $\II(Y_2,L_2,\omega_2)$ (up to sign)  as follows. Let $(W,S,\Omega)$ be the cobordism obtained by gluing $[-1,0] \times (Y_1,L_1,\omega_1)$ with $[0,1]\times (Y_2,L_2,\omega_2)$ via the map 
$$\id\times \varphi:\{0\}\times (Y_1,L_1,\omega_1)\to \{0\}\times (Y_2,L_2,\omega_2),$$ then the map induced by $\varphi$ is defined to be the cobordism map $\II(W,S,\Omega)$. 

If $\varphi$ is the composition of two diffeomorphisms $\varphi_1$ and $\varphi_2$, then the isomorphism defined by $\varphi$ on instanton homology is equal to the composition of the isomorphisms defined by $\varphi_1$ and $\varphi_2$ up to sign.

We will be mostly interested in the case when $\varphi$ is a self-diffeomorphism.
If $\varphi$ is homotopic to the identity through self-diffeomorphisms of $(Y,L,\omega)$, then the map induced by $\varphi$ is $\pm \id$. This is because the cobordism that defines this map is diffeomorphic to the trivial product via a diffeomorphism that fixes the boundary.
\end{exmp}

If $L$ is oriented and $\partial \omega=\emptyset$, there is a special class of auxiliary data such that for all $a_1,a_2$ in this class, the isomorphism from $\II(Y,L,\omega,a_1)$ to  $\II(Y,L,\omega,a_2)$ is canonically defined without sign ambiguity\footnote{This is achieved by specifying a canonical reducible connection on the bundle. For details, the reader may refer to \cite{KM:Kh-unknot}, see the paragraph containing Equation (12) and the discussion after Definition 3.7.}. 
In this case, we will always assume that the auxiliary data is chosen in such special class.
If $\varphi:(Y_1,L_1,\omega_1)\to (Y_2,L_2,\omega_2)$ is a diffeomorphism preserving the orientations of $Y_i$ and $L_i$, then the map $\varphi$ induces an isomorphism from $\II(Y_1,L_1,\omega_1)$ to $\II(Y_2,L_2,\omega_2)$ without sign ambiguity.

Suppose $(W,S,\Omega)$ is a cobordism from $(Y_1,L_1,\omega_1)$ to $(Y_2,L_2,\omega_2)$, where $L_i$'s are oriented and $\partial \omega_i =\emptyset$.
If $W$ is endowed with an almost complex structure such that $S$ is an almost complex submanifold and $\Omega\cap S=\emptyset$, there is a canonical choice of sign for the cobordism map $\II(W,S,\Omega)$, see \cite[Section 5.1]{KM:Kh-unknot}. When such cobordisms are glued in a way that is compatible with the almost complex structures, the functoriality property holds without sign ambiguity.

\begin{rem}
There is another setting where the sign of the cobordism map can be fixed. Suppose $(W,S,\Omega)$ is a cobordism from $(Y_1,L_1,\omega_1)$ to $(Y_2,L_2,\omega_2)$, where $L_i$ are oriented and $\partial \omega_i =\emptyset$.
If $Y_1\cong Y_2$, $W\cong [-1,1]\times Y_1$, $\Omega\cap S=\emptyset$, and $S$ is oriented, then there is a canonical choice of sign for the cobordism map $\II(W,S,\Omega)$. This construction is not needed in our paper. 
\end{rem}

From now, all instanton homology groups will be defined with $\bC$ coefficients, and we will omit the coefficient ring from the notation. 

If $(Y,L,\omega)$ is given by the disjoint union of $(Y_1,L_1,\omega_1)$ and $(Y_2,L_2,\omega_2)$, then the K\"unneth formula (for $\bC$ coefficients) gives a canonical isomorphism 
$$
\II(Y,L,\omega) \cong \II(Y_1,L_1,\omega_1)\otimes  \II(Y_2,L_2,\omega_2).
$$


\subsection{$\mu$ maps and excision}
Suppose $(Y,L,\omega)$ is an admissible triple and $a$ is a fixed choice of auxiliary data.
For $\alpha\in H_*(Y;\bZ)$, there is a homomorphism 
$$
\muu(\alpha):\II(Y,L,\omega,a)\to \II(Y,L,\omega,a)
$$
 on singular instanton homology.
The $\mu$ map for instanton Floer homology was first introduced in  \cite{donaldson1990polynomial, D-Floer};
for the definition of the $\mu$ map in the context of singular instanton Floer homology, the reader may refer to \cite[Section 2.3.2]{Street} and \cite[Section 2.1]{XZ:excision}.
There is a choice of an overall multiplicative constant factor in the definition of $\mu$ maps, and we follow the convention in \cite[Section 2.1]{XZ:excision} and choose the constant to be $-1/4$. 

Once the auxiliary data is fixed, the $\muu$ map is defined without sign ambiguity. Suppose $a_1$ and $a_2$ are two choices of auxiliary data for $(Y,L,\omega)$ and let $\varphi:\II(Y,L,\omega,a_1)\to \II(Y,L,\omega,a_2)$ be the canonical isomorphism. The map $\varphi$ is only defined up to an overall sign, and we choose an arbitrary sign for $\varphi$. Then $\varphi\circ \muu(\alpha) = \muu(\alpha)\circ \varphi$ and the equation holds without sign ambiguity. 

If $\alpha$ is contained in the push-forward of $H_*(Y\backslash L;\bZ)$ to $H_*(Y;\bZ)$, it is also customary in the literature to write $\muu(\alpha)$ as $\mu(\alpha)$.

The $\muu$ maps satisfy the following properties.
\begin{property}
	\label{property_mu_map_intertwine_with_cobordism}
	Suppose $(W,S,\Omega)$ is a cobordism between two admissible triples $(Y_1,L_1,\omega_1)$ and $(Y_2,L_2,\omega_2)$. Let $\alpha_1\in H_*(Y_1;\bZ)$, $\alpha_2\in H_*(Y_2;\bZ)$ be two homology classes such that their images in $H_*(W;\bZ)$ are the same. Fix a choice of sign for $\II(W,S,\Omega)$, the following equation holds without sign ambiguity.
		$$
		\II(W,S,\Omega)\circ \muu(\alpha_1) =  \muu(\alpha_2)\circ \II(W,S,\Omega).
		$$ 
\end{property}
\begin{cor}
	\label{cor_mu_map_eigenspace_preserved_by_cobordism}
	Let $(Y_1,L_1,\omega_1)$, $(Y_2,L_2,\omega_2)$, $(W,S,\Omega)$, $\alpha_1$, $\alpha_2$ be as in Property \ref{property_mu_map_intertwine_with_cobordism}.
	Then for every $\lambda\in \bC$, the map $\II(W,S,\Omega)$ sends the generalized eigenspace of $\muu(\alpha_1)$ in $\II(Y_1,L_1,\omega_1)$ with eigenvalue $\lambda$ to the generalized eigenspace of $\muu(\alpha_2)$ in $(Y_2,L_2,\omega_2)$ with eigenvalue $\lambda$.
\end{cor}
\begin{property}
	\label{property_linearity_and_commutativity_of_mu}
	\begin{enumerate}
\item 
For all $\alpha,\alpha'\in  H_*(Y;\bZ)$, we have 
$$
\muu(\alpha+\alpha') = \muu(\alpha) + \muu(\alpha').
$$
\item For $\alpha\in H_k(Y;\bZ)$, $\alpha'\in H_{k'}(Y;\bZ)$, we have
$$
\muu(\alpha) \muu(\alpha')= (-1)^{kk'} \muu(\alpha')	\muu(\alpha).
$$
\end{enumerate}
\end{property}

It is often useful to represent $\alpha$ by embedded manifolds in $Y$. To simplify notation, if $\Sigma$ is a closed oriented surface embedded in $Y$, we will write $\muu([\Sigma])$ as $\muu(\Sigma)$. Similarly, if $p\in Y$ is a point, we will write $\muu([p])$ as $\muu(p)$. 

 If $\alpha\in H_{k}(Y;\bZ)$, then $\muu(\alpha)$ has degree $4-k$ with respect to the relative $\bZ/4$ grading. This observation and an elementary argument lead to the following lemma.

\begin{lem}[{\cite[Lemma 3.14]{xie2020instantons}}]
	\label{lem_same_dim_opposite_eigenvalues}
	 Suppose $\alpha\in H_{2}(Y;\bZ)$. Then for every $\lambda\in \bC$, the dimension of the generalized eigenspace of $\muu(\alpha)$ on $\II(Y,L,\omega)$ with eigenvalue $\lambda$ is the same as the dimension of the generalized eigenspace with eigenvalue $-\lambda$. 
\end{lem}

We will also need the following result about the eigenvalues of $\mu$ maps.

\begin{prop}[{\cite[Corollary 7.2]{KM:suture}}, {\cite[Proposition 6.1]{XZ:excision}}]
	\label{prop_eigenvalues_bounded_by_Thurston_norm} 
Let $\II^{(2)}(Y,L,\omega)$ be the simultaneous generalized eigenspace of $\muu(p)$ with eigenvalue $2$ for all $p\in Y$. Suppose $\Sigma$ is a non-integral surface with genus $g$ that intersects $L$ transversely at $n$ points such that $2g-2+n\ge 0$, then all eigenvalues of $\muu(\Sigma)$ on $\II^{(2)}(Y,L,\omega)$ are contained in the set 
		$$\{2g-2+n-2k\,|\,k\in \bZ, \, 0\le k\le 2g-2+n\}.$$ 
\end{prop}

The above proposition motivated the following definition in \cite{XZ:excision}.
\begin{defn}
	Let $(Y,L,\omega)$ be an admissible triple. 
Suppose $\Sigma$ is a closed oriented embedded surface in $Y$. Let $\Sigma_1,\dots,\Sigma_l$ be the connected components of $\Sigma$. Assume $\Sigma_i$ has genus $g_i$ and intersects $L$ transversely at $n_i$ points, and assume that $2g_i-2+n_i\ge 0$ for all $i$. Define $\II(Y,L,\omega|\Sigma)\subset \II^{(2)}(Y,L,\omega)$ to be the simultaneous generalized eigenspace of 
$$\muu(\Sigma_1),\dots,\muu(\Sigma_l)$$ with eigenvalues $$2g_1-2+n_1,\dots,2g_l-2+n_l.$$
\end{defn}

Note that in the above definition, we do not require $\Sigma_i$ to be non-integral. This turns out to be convenient in some occasions. 

If $\Sigma$ is a torus that is disjoint from $L$ and intersects $\omega$ transversely at $1$ point, then by Proposition \ref{prop_eigenvalues_bounded_by_Thurston_norm} above, we have $\II(Y,L,\omega|\Sigma)= \II^{(2)}(Y,L,\omega)$.

The group $\II(Y,L,\omega)$ satisfies the following property, which is usually referred to as the excision theorem.

\begin{prop}[{\cite[Theorem 7.7]{KM:suture}, \cite[Proposition 6.4]{XZ:excision}}]
	\label{prop_excision}
	Suppose $(Y,L,\omega)$ is an admissible triple. 
	Suppose $\Sigma_1,\Sigma_2\subset Y$ are two disjoint non-integral surfaces such that both of them have genus $g$,  intersect $L$ transversely at $n$ points, are disjoint from $\partial \omega$, and intersect $\omega$ transversely at $m$ points. Suppose either $n=0$ and $m$ is odd, or $n$ is odd and $n\ge 3$. Let $\varphi:\Sigma_1\to\Sigma_2$ be an orientation-preserving diffeomorphism that maps $\Sigma_1\cap L$ to $\Sigma_2\cap L$ and maps $\Sigma_1\cap \omega$ to $\Sigma_2\cap \omega$. Write $\Sigma_1\cup \Sigma_2$ as $\Sigma$. Let $(\widetilde{Y}, \widetilde{L}, \tilde{\omega})$ be the resulting triple after cutting $Y$ open along $\Sigma=\Sigma_1\cup \Sigma_2$ and gluing back the boundary components via the map $\varphi$, let $\widetilde{\Sigma}\subset\widetilde{Y}$ be the image of $\Sigma$ after gluing.  Then the above excision defines a cobordism $(W,S,\Omega)$ from $(Y,L,\omega)$ to $(\widetilde{Y},\widetilde{L},\tilde{\omega})$, which induces an isomorphism from 
$\II(Y,L,\omega|\Sigma)$ to  $\II(\widetilde{Y},\widetilde{L},\tilde{\omega}| \widetilde{\Sigma}).$
\end{prop}

We will also need the following computation of instanton Floer homology. Recall that $t_*\in S^1$ is the image of $\{-1,1\}\subset[-1,1]$ under the quotient map.
\begin{prop}[{\cite[Proposition 7.8]{KM:suture}, \cite[Proposition 6.7]{XZ:excision}}]
		\label{prop_product_space_homology_rank_1}
	Suppose $\Sigma$ is a closed oriented surface with genus $g$, let $p_1,\dots,p_n, q_1,\dots,q_m$ be distinct points on $\Sigma$. 
	Let $(u,\partial u)\subset (\Sigma, \{p_1,\dots,p_n\})$ be an embedded compact $1$--manifold in $\Sigma$ whose interior is disjoint from $\{p_1,\dots,p_n,q_1,\dots,q_m\}$. 
	Assume $2g-2+n\ge 0$, and assume that either $n=0$ and $m$ is odd, or $n$ is odd and $n\ge 3$. Then we have 
	$$\II(S^1\times \Sigma, S^1\times \{p_1,\dots,p_n\}, S^1\times \{q_1,\dots,q_m\}\cup \{t_*\}\times u|\{t_*\}\times \Sigma)\cong \bC.$$
\end{prop}

\subsection{Annular instanton Floer homology}

Recall that $A= [-1,1]\times S^1$ denotes the oriented annulus. Suppose $L$ is a link in the interior of $[-1,1]\times A$. The \emph{annular instanton homology} of $L$ was introduced in  \cite{AHI}, and it is defined by the following steps. 
\begin{enumerate}
	\item Fix an orientation on the standard $S^2$. Write $S^2$ as the union of two disks $D_1\cup D_2$ with $D_1\cap D_2\cong S^1$.
	\item Embed $[-1,1]\times [-1,1]$ into $D_1$ by a fixed orientation-preserving map.  
	\item View $L\subset [-1,1]\times A = [-1,1]\times [-1,1]\times S^1$ as a link in $D_1\times S^1$ via the above embedding.
	\item Take two points $p,q$ in $D_2$, and let $c$ be an arc connecting $p$ and $q$ in $D_2$.
	\item Define $\AHI(L)$ to be the singular instanton homology 
	$$\II\big(S^2\times S^1, L\cup (\{p,q\}\times S^1), c\times\{t_*\}\big).$$
\end{enumerate}

Let $\Sigma = S^2\times\{t_*\}\subset S^2\times S^1$ be endowed with the same orientation as the standard $S^2$. It was proved by 
\cite[Proposition 4.5]{AHI}
that the eigenvalues of $\muu(\Sigma)$ on $\AHI(L)$ are all integers. The decomposition of $\AHI(L)$ by the generalized eigenspaces of $\muu(\Sigma)$ defines a grading on $\AHI(L)$, which is called the f-grading. By Lemma \ref{lem_same_dim_opposite_eigenvalues} (and also by \cite[Proposition 4.5]{AHI}), for each $i\in \bZ$, the component of $\AHI(L)$ with f-degree $i$ has the same dimension as the component of $\AHI(L)$ with degree $-i$. 

Suppose $L_1$ and $L_2$ are two links in the interior of $[-1,1]\times A$ and $S$ is a link cobordism from $L_1$ to $L_2$, then $S$ defines a cobordism map from $\AHI(L_1)$ to $\AHI(L_2)$, and we denote the cobordism map by $\AHI(S)$. 

\subsection{Sutured instanton Floer homology for tangles}
We will also need the theory of sutured instanton Floer invariant for tangles which was introduced in \cite{XZ:excision}. Suppose $(M,\ga)$ is a balanced sutured manifold and $T$ is a balanced tangle in $(M,\ga)$, we will use $\SHI(M,\ga,T)$ to denote the sutured instanton Floer homology of $(M,\ga,T)$ as defined in \cite[Definition 7.6]{XZ:excision}. For the definition of balanced sutured manifolds and balanced tangles, we refer the reader to  \cite[Section 7]{XZ:excision}. When $T=\emptyset$, it is customary to denote $\SHI(M,\ga,\emptyset)$ as $\SHI(M,\ga)$. 

We remark here that $\SHI(M,\ga,T)$ is only defined to be an \emph{isomorphism class} of $\bC$--vector spaces instead of an actual vector space. The reader may refer to \cite{baldwin2015naturality} for a general discussion of the naturality problem of sutured homology.

We will be mostly interested in the case when $(M,\ga)=([-1,1]\times \Sigma,\{0\}\times \partial \Sigma)$, where $\Sigma$ is a compact oriented surface with boundary, and $T$ is an embedded link in the interior of $M$. In this case, there is an explicit description of $\SHI(M,\ga,T)$ given as follows. 
\begin{prop}
	\label{prop_SiHI_iso_SHI}
	Let $\Sigma$, $T$ be as above.
Let $F$ be a compact surface such that there is an orientation-reversing diffeomorphism $\tau:\partial F\to \partial \Sigma$. We also assume that $F$ does not have disk components. Let $R=F\cup_\tau \Sigma$, let $p$ be a point on $F$, let $c$ be an arbitrary simple closed curve on $R$ such that there exists a simple closed curve $c'\subset F$ intersecting $c$ transversely at one point. Then we have 
\begin{align}
&\SHI([-1,1]\times \Sigma,\{0\}\times \partial \Sigma,T)
	\label{eqn_SiHI_iso_SHI_0}
\\
\cong &\II(S^1\times R, T, S^1\times \{p\}|\{t_*\}\times R) 
	\label{eqn_SiHI_iso_SHI_1}
\\
\cong &\II(S^1\times R, T, \{t_*\}\times c|\{t_*\}\times R).
\label{eqn_SiHI_iso_SHI_2}
\end{align}
where $T$ is viewed as a subset of $S^1\times R$ via the embedding of $(-1,1)\times \Sigma$ to $S^1\times R$.
\end{prop}

Proposition \ref{prop_SiHI_iso_SHI} follows from standard excision arguments. 
We briefly review the proof here by listing a few references. When $F$ is connected, the fact that \eqref{eqn_SiHI_iso_SHI_1} and \eqref{eqn_SiHI_iso_SHI_2} are isomorphic and that their isomorphism class does not depend on the choice of $F$ follows verbatim from the proof of \cite[Theorem 4.4]{KM:suture}. The case when $F$ is disconnected follows from the discussion  in \cite[Theorem 4.4]{KM:suture}.
The fact that \eqref{eqn_SiHI_iso_SHI_1} and \eqref{eqn_SiHI_iso_SHI_2} are both isomorphic to \eqref{eqn_SiHI_iso_SHI_0} follows from \cite[Remark 7.8]{XZ:excision} when $L=\emptyset$, and the same argument can be extended verbatim to the case when $L$ is a link in the interior of $M$.

\section{Proof of Theorem \ref{thm_main_instanton}}
\label{sec_two_nonvanishing_grading}

The main purpose of this section is to establish Theorem \ref{thm_main_instanton}. The arguments in this section rely heavily on the techniques of sutured manifold decomposition. We will use the notation from \cite{Schar} for sutured manifolds and sutured hierarchies. 

\subsection{Product sutured manifold and gradings}\label{subsection: surfaces induce gradings on SHI}
Suppose $\Sigma$ is a connected oriented surface with boundary.
Let $(M,\ga)$ be the sutured manifold given by
$$(M,\ga)=([-1,1]\times \Sigma,\{0\}\times \partial \Sigma).$$
Suppose $L$ is a link in the interior of $M$. Suppose $\al$ is an oriented properly embedded $1$--submanifold of $\Sigma$. We construct a ``grading'' on $\SHI(M,\ga,L)$ associated to $\al$ as follows: Pick a connected oriented surface $F$ so that there is an orientation reversing diffeomorphism
$$\tau:\partial F\to \ga,$$
and assume that the genus of $F$ is positive.
Pick a point $p\in F$ and an oriented properly embedded $1$--submanifold $\al'$ on $F\backslash\{p\}$ so that $\tau$ restricts to an orientation reversing diffeomorphism
$$\tau:\partial \al'\to \partial\al.$$
To construct the grading, write $R=\Sigma\cup_\tau F$ and recall that by Proposition \ref{prop_SiHI_iso_SHI},
$$\SHI(M,\ga,L)\cong \II(S^1\times R, L, S^1\times \{p\}|\{t_*\}\times R).$$
Now we define
\begin{equation}\label{eq: grading associated to al}
	\SHI(M,\ga,L,\al,i)={\rm Eig}(\muu(S^1\times(\al\cup\al')),i)\cap \II(S^1\times R, L, S^1\times \{p\}|\{t_*\}\times R),
\end{equation}
where the intersection is taken as subspaces of $\II(S^1\times R, L, S^1\times \{p\})$,
and ${\rm Eig}(\muu(S^1\times(\al\cup\al'),i))$ denotes the generalized eigenspace of the map $\muu(S^1\times(\al\cup\al'))$ with eigenvalue $i$.
\blem\label{lem: grading is well defined}
The isomorphism class of $\SHI(M,\ga,L,\al,i)$ is independent of the choice of $F$ and $\al'$. 
\elem
\bpf
The proof that the isomorphism class of $\SHI(M,\ga,L,\al,i)$ remains the same when $F$ is changed to $F\# T^2$ follows from the same argument as \cite[Theorem 4.4]{KM:suture}. 

To show that it is independent of the choice of $\al'$, we first claim that if $\beta$ is a non-separating simple closed curve on $R$ such that $p\notin \beta$ and $(S^1\times \be) \cap L = \emptyset$, then the action of $\muu(S^1\times \be)$ on 
$$\II(S^1\times R, L, S^1\times \{p\}|\{t_*\}\times R)$$ 
is nilpotent. To prove the claim, let $\be'$ be a simple closed curve on $R$ intersecting $\be$ transversely at one point. Then by  the excision theorem (Proposition \ref{prop_excision}), we have 
\begin{align*}
&\II(S^1\times R, L, S^1\times \{p\}|\{t_*\}\times R)  
\\
\cong &\II(S^1\times R, L, S^1\times \{p\}\cup \{t_*\}\times \be'|\{t_*\}\times R) \otimes \II(S^1\times R, \emptyset, S^1\times \{p\}\cup \{t_*\}\times \be'|\{t_*\}\times R).
\end{align*}
The above isomorphism is defined by an excision cobordism map. 
By Proposition \ref{prop_eigenvalues_bounded_by_Thurston_norm}, the action of $\muu(S^1\times \be)$ is nilpotent on $\II(S^1\times R, L, S^1\times \{p\}\cup \{t_*\}\times \be'|\{t_*\}\times R)$ and $ \II(S^1\times R, \emptyset, S^1\times \{p\}\cup \{t_*\}\times \be'|\{t_*\}\times R)$. Hence by Property \ref{property_mu_map_intertwine_with_cobordism}, the action of $\muu(S^1\times \be)$ is also nilpotent on $\II(S^1\times R, L, S^1\times \{p\}|\{t_*\}\times R)$, and the claim is proved.

Now suppose $\al''$ is another properly embedded oriented manifold in $F$ such that  $\partial \al'' = \partial \al'$ as oriented manifolds. Then 
$$[S^1\times (\al\cup\al')]=[S^1\times (\al\cup\al'')]+[S^1\times(\al'\cup-\al'')]\in H_2(S^1\times R).$$
Note that $S^1\times(\al'\cup-\al'')$ is homologous to a union of tori of the form $\be_i\times S^1$, where each $\be_i$ is a non-separating simple closed curve on $R$ such that  $p\notin \beta_i\subset F$. By the previous discussion, we conclude that $\muu([S^1\times(\al'\cup-\al'')])$ is nilpotent. Therefore the desired result follows from Property \ref{property_linearity_and_commutativity_of_mu}.
\epf

By \cite[Lemma 3.16]{xie2020instantons} (see also Lemma \ref{lem_grading_SiHI_range} below), all eigenvalues of $\muu(S^1\times(\al\cup\al'))$ on $\II(S^1\times R, L, S^1\times \{p\}|\{t_*\}\times R)$ are integers. Therefore we have

\begin{lem}
	\label{lem_sum_gradings_SHI}
	$\dim \SHI(M,\ga,L) = \sum_{i\in \bZ} \dim \SHI(M,\ga,L,\al,i). $ \qed
\end{lem}
We view the groups $\SHI(M,\ga,L,\al,i)$ ($i\in \bZ$) as a ``grading'' on $\SHI(M,\ga,L)$ in the sense that they define a partition of the dimension of $\SHI(M,\ga,L)$.

\bdefn
\label{def_groomed}
We say an oriented properly embedded $1$-submanifold $\al\subset\Sigma$ is {\it groomed} if it satisfies either one of the following two conditions.
\begin{enumerate}
	\item [(1)] $\al$ has no closed component, and for every component $\sigma$ of $\partial \Sigma$, all intersections between $\al$ and $\sigma$ are of the same sign.
	\item [(2)] Every component of $\al$ is closed, and $\al$ consists of parallel, coherently oriented non-separating simple closed curves.
\end{enumerate}
\edefn

\brem
The term \emph{groomed} first appeared in \cite{G:Sut-2} to describe surfaces in $(M,\ga)$. The $1$--manifold $\al$ is groomed in the sense of Definition \ref{def_groomed} if and only if $[-1,1]\times \al$ is a groomed surface in $([-1,1]\times \Sigma,\{0\}\times \partial \Sigma)$ in the sense of \cite{G:Sut-2}. 
\erem

The main result of this section is the following. 
\bprop\label{prop: two non-zero gradings}
Suppose $\Sigma$ is a compact connected oriented surface with non-empty boundary, and suppose  $\al\subset \Sigma$ is groomed. Let
$$(M,\ga)=([-1,1]\times \Sigma,\{0\}\times \partial \Sigma).$$
Suppose $L\subset {\rm int}(M)$ is a non-empty link so that 
\begin{enumerate}
	\item $(M,\ga)$ is $L$--taut, 
	\item every product annulus in $M\backslash L$ is trivial (cf. \cite[Definition 4.1]{Schar}), 
	\item  every product disk in $M\backslash L$ is boundary parallel (cf. \cite[Definition 4.1]{Schar}),
\end{enumerate}
 then there are integers $i_+\neq i_-$ so that 
$$\SHI(M,\ga,L,\al,i_+)\neq0{\rm~and~}\SHI(M,\ga,L,\al,i_-)\neq0.$$
\eprop
\bpf
If $L$ can be isotoped to a disjoint union of non-empty links $L_1\subset (0,1)\times \Sigma$ and  $L_2\subset (-1,0)\times \Sigma$, then 
$(M,\ga)$ is both $L_1$--taut and $L_2$--taut. By \cite[Theorem 7.12]{XZ:excision}, $\SHI(M,\ga,L_1)$ and $\SHI(M,\ga,L_2)$ are both non-vanishing. 
Hence by the excision theorem and Property \ref{property_mu_map_intertwine_with_cobordism}, one only needs to prove the desired property for $L_1$ and $L_2$. As a result, we may assume without loss of generality that $L$ cannot be isotoped to a disjoint union of $L_1$ and $L_2$ as above. This is equivalent to the statement that there is no non-trivial horizontal surfaces in $(M,\ga)$ disjoint from $L$ (cf. \cite[Definition 2.17]{juhasz2010polytope}).

Let $S_{\al}=[-1,1]\times \al$. Let $N\subset \partial M$ be the union of $[-1,1]\times \partial \Sigma$ and a small tubular neighborhood of $\partial S_\al$.
Applying \cite[Theorem 2.5]{Schar} to $S_\al$ itself and to $S_\al$ with the opposite orientation, we obtain two surfaces $S_+$ and $S_-$ inside $M$ so that the following conditions hold:
\begin{itemize}
	\item [(i)] $\partial S_\pm \subset N$, the intersection $\partial S_\pm\cap [-1,1]\times \partial \Sigma$ only contains essential circles and essential arcs in $[-1,1]\times \partial \Sigma$, and $\partial S_\pm \cap R(\ga) = (\pm S_\al) \cap R(\ga)$. 
	\item [(ii)] there exist positive integers $k_{\pm}$ so that  
	\begin{equation}
		\label{eqn_homology_S_pm}
	 [S_{\pm},\partial S_{\pm}]=\pm[S_{\al},\partial S_{\al}]+k_{\pm}\cdot [R_+(\al),\partial R_+(\al)]\in H_2(M,N),
	\end{equation}
	\item [(iii)] $S_+\cap R_{\pm}(\ga)=-S_-\cap R_{\pm}(\ga)=\{\pm1\}\times\al.$
	\item [(iv)] There are taut sutured manifold decompositions
	$$(M,\ga,L)\stackrel{S_+}{\leadsto}(M_+,\ga_+,T_+)~{\rm and~}(M,\ga,L)\stackrel{S_-}{\leadsto}(M_-,\ga_-,T_-).$$
\end{itemize}

\brem
In the original statement of \cite[Theorem 2.5]{Schar}, Equation \eqref{eqn_homology_S_pm} was replaced with the weaker statement
$$[S_{\pm},\partial S_{\pm}]=\pm[S_{\al},\partial S_{\al}]+k_{\pm}\cdot [R_+(\al),\partial R_+(\al)]\in H_2(M,\partial M),$$
yet the proof in \cite{Schar} actually constructed surfaces $S_{\pm}$ satisfying the stronger property. This observation was also utilized in the proof of \cite[Theorem 6.1]{juhasz2010polytope}.
\erem



Note the arc components of $S_+\cap [-1,1]\times \partial\Sigma$ come from the intersection points in $\partial\al\cap \partial\Sigma$. The assumption that $\al$ is groomed then implies that all arc intersections of $S_+$ with a given component of $[-1,1]\times\partial\Sigma$ are parallel and coherently oriented. For closed components of $S_+\cap [-1,1]\times\partial \Sigma$, if two such components are oriented oppositely, we can glue an annulus to them to remove both components without affecting the properties listed above. As a result, we can assume that for any component $\sigma$ of $\partial \Sigma$, the intersection $S_+\cap[-1,1]\times\sigma$ consists of parallel and coherently oriented essential simple closed curves or arcs. The same statement applies to $S_-$ as well. By \eqref{eqn_homology_S_pm}, we have 
$$[\partial S_{\pm}]=[\pm\partial S_{\al}]+k_{\pm}\cdot [\ga]\in H_1(N).$$
By Condition (i) above, we have $\partial S_{\pm}\cap R(\ga)=\{-1,1\}\times (\pm\al)$, so $\partial S_{\pm}$ can be obtained from the union of $\pm\partial S_{\al}$ and $k_{\pm}$ copies of $\ga$ by oriented resolutions on their intersections.

Suppose $F$ is an auxiliary surface for $(M,\ga)$ and $\al'\subset F$ is chosen as in the construction of the grading associated to $\al$. Write
$$R=\Sigma\cup F~{\rm and}~\widetilde{S}_{\al}=[-1,1]\times(\al\cup\al').$$

 We extend $S_{\pm}$ to closed surfaces inside $S^1\times R$ as follows. Inside $[-1,1]\times F$, we perform a cut and paste surgery on $k_{\pm}$ parallel copies of $F$ and $[-1,1]\times\al'$, i.e., we cut these surfaces open along intersections and reglue in an orientation preserving way to obtain a properly embedded surface. Then we glue the resulting surface to $S_{\pm}$ via the identification of
 $[-1,1]\times\partial \Sigma$ and $[-1,1]\times F.$
 Let $\widetilde{S}_{\pm}$ be the resulting properly embedded surface inside $[-1,1]\times R$. Then we have
 $$\partial \widetilde{S}_{\pm}\cap\{-1,1\}\times R=\pm\widetilde{S}_{\al}\cap\{-1,1\}\times R=\{-1,1\}\times(\al\cup\al').$$
 As a result, when we glue $\{1\}\times R$ to $\{-1\}\times R$, the surfaces $\widetilde{S}_{\pm}$ are glued to become closed surfaces $\widebar{S}_{\pm}$ inside $S^1\times R$. 
 Let $\widebar{S}_\al = S^1\times (\al\cup \al').$
 Condition (ii) for $S_{\pm}$ and our construction of $\widebar{S}_{\pm}$ then imply that 
\begin{equation}\label{eq: relation between S_pm and S_al}
	[\widebar{S}_{\pm}]=\pm[\widebar{S}_{\al}]+k_{\pm}\cdot [R]\in H_2(S^1\times R).
\end{equation}

Recall that by definition, the action of $\muu(\{t_*\}\times R)$ on $\II(S^1\times R,L,S^1\times\{p\}|\{t_*\}\times R)$ has only one eigenvalue $2g(R)-2$. By Condition (iv) above and \cite[Theorem 2.14]{LXZ-unknot}, we have
\beq
0&\neq \SHI(M_{\pm},\ga_{\pm},T_{\pm})\\
&\cong \II(S^1\times R,L,S^1\times\{p\}|\{t_*\}\times R)\cap\Eig(\mu^{orb}(\widebar{S}_{\pm}),|S_{\pm}\cap L|-\chi(\widebar{S}_{\pm}))\\
&\cong \II(S^1\times R,L,S^1\times\{p\}\{t_*\}\times R)\cap\Eig(\mu^{orb}(\widebar{S}_{\al}),|S_{\pm}\cap L|-\chi(\widebar{S}_{\pm})+k_{\pm}\cdot \chi(R))\\
&=\SHI\bigg(M,\ga,L,\al,\pm\big(|S_{\pm}\cap L|-\chi(\widebar{S}_{\pm})+k_{\pm}\cdot \chi(R)\big)\bigg).
\eeq
Hence if we take
$$i_{\pm}=\pm\big(|S_{\pm}\cap L|-\chi(\widebar{S}_{\pm})+k_{\pm}\cdot \chi(R)\big),$$
then 
$$\SHI(M,\ga,L,\al,i_{\pm})\neq0.$$
It remains to show that $i_+\neq i_-$. If $i_+=i_-$, then we have
\begin{align}
&(|S_{+}\cap L|-\chi(\widebar{S}_{\pm})+k_+\cdot \chi(R))=-(|S_{-}\cap L|-\chi(\widebar{S}_{-})+k_-\cdot \chi(R))
\nonumber \\
\Rightarrow~&|S_+\cap L|+|S_-\cap L|-\chi(\widebar{S}_{+})-\chi(\widebar{S}_{-})=-(k_++k_-)\cdot \chi(R)
\nonumber \\
\Rightarrow~&|S_+\cap L|+|S_-\cap L|-\chi({S}_{+})-\chi({S}_{-})+|S_{+}\cap \ga|=-(k_++k_-)\cdot \chi(R_+(\ga)).
\label{eqn_i+=i-_implied}
\end{align}
Where in the last step, we used the construction of $\widebar{S}_{\pm}$ to deduce that
$$\chi(\widebar{S}_{\pm})=\chi(S_{\pm})-\frac{1}{2}|S_{\pm}\cap \ga|+k_{\pm}\cdot \chi(F);$$
also by construction, we have $|S_+\cap\ga|=|S_-\cap\ga|$. Equation \eqref{eqn_i+=i-_implied} directly contradicts Lemma \ref{lem: juhasz's ineq} below. Hence the proof is finished.
\epf

We postpone the proof of Lemma \ref{lem: juhasz's ineq} to Section \ref{subsec: key inequality}. In the next section, we use Proposition \ref{prop: two non-zero gradings} to prove Theorem \ref{thm_main_instanton}.

\subsection{Proof of Theorem \ref{thm_main_instanton} assuming Proposition \ref{prop: two non-zero gradings}}

\bprop
\label{prop_SHI_rank_greater_than_2_if_no_trivial_product}
Suppose $\Sigma$ is a compact connected oriented surface with non-empty boundary. Assume $\Sigma$ satisfies at least one of the following two conditions:
\begin{enumerate}
	\item $\Sigma$ has at least three boundary components.
	\item $\Sigma$ has genus at least one.
\end{enumerate} 
Let $(M,\ga)=([-1,1]\times \Sigma,\{0\}\times \partial \Sigma).$
 Suppose $L\subset {\rm int}(M)$ is a link so that $(M,\ga)$ is $L$--taut, every product annulus in $M\backslash L$ is trivial, and every product disk in $M\backslash L$ is boundary parallel.
Then $$\dim \SHI(M,\ga,L)>2.$$
\eprop

\bpf
Assume  $\dim \SHI(M,\ga,L)\le 2$. Take two oriented $1$--manifolds $\al_1$ and $\al_2$ of $\Sigma$ as follows:
\begin{enumerate}
	\item If $\Sigma$ has at least three boundary components, let $\al_1$ and $\al_2$ be  two properly embedded arcs so that $\al_1\cap\al_2=\emptyset$ and $\partial \al_1\cup\partial \al_2$ intersects at least three boundary components of $\partial\Sigma$.
	\item Otherwise, $\Sigma$ has genus at least one, and we let $\al_1$ and $\al_2$ be a pair of simple closed curves such that $\al_1$ and $\al_2$ intersect transversely at one point.
\end{enumerate}
By construction, $\al_1$ and $\al_2$ are both groomed. Then by Proposition \ref{prop: two non-zero gradings}, for $j=1,2$, there are distinct integers $i_{j,+}\neq i_{j,-}$ so that 
$$\SHI(M,\ga,L,\al_j,i_{j,+})\neq0{\rm~and~}\SHI(M,\ga,L,\al_j,i_{j,-})\neq0.$$

Note that we can choose suitable auxiliary data so that $S_{\al,j}=[-1,1]\times \al_j$ extends to closed surfaces $\widebar{S}_{\al,j}$ inside the same closure $Y=S^1\times (\Sigma\cup F)$ for both $j$. Since the actions $\muu(\widebar{S}_{\al,1})$ and $\muu(\widebar{S}_{\al,2})$ commute with each other, after changing the orders of $i_{j,+}$ and $i_{j,-}$ if necessary, we may assume that
$$\SHI(M,\ga,L,(\al_1,\al_2),(i_{1,+},i_{2,+}))\cong \SHI(M,\ga,L,(\al_1,\al_2),(i_{1,-},i_{2,-}))\cong\mathbb{C},$$
where $\SHI(M,\ga,L,(\al_1,\al_2),(\lambda_1,\lambda_2))$ denotes the intersection of 
$$\SHI(M,\ga,L,\al_1,\lambda_1) ~\text{and}~ \SHI(M,\ga,L,\al_2,\lambda_2)$$
 as subspaces of $\II(S^1\times (\Sigma\cup F), L, S^1\times \{p\}|\{t_*\}\times (\Sigma\cup F))$.

Pick co-prime integers $r$ and $s$ so that
$$r\,( i_{1,+}- i_{1,-})=s\,( i_{2,-}- i_{2,+}).$$
If $r>0$, let $r\cdot \al_1$ be $r$ parallel copies of $\al_1$; if $r<0$, let $r\cdot \al_1$ be $|r|$ copies of $-\al_1$. Define $s\cdot \al_2$ similarly. 
If $\al_1$ and $\al_2$ are arcs, let $\al_3$ be the union of $r\cdot \al_1$ and $s\cdot \al_2$.
If $\al_1$ and $\al_2$ are simple closed curves, let $\al_3$ be the union of $r\cdot \al_1$ and $s\cdot \al_2$ with singularities resolved at all intersections.
Then
\begin{equation}
	\label{eq: only one grading}
	\SHI(M,\ga,L,\al_3,i_3)\cong\mathbb{C}^2\cong \SHI(M,\ga,L).
\end{equation}

Note that if $\al_1$ and $\al_2$ are simple closed curves that intersect transversely at one point, then $\al_3$ is a non-separating simple closed curve on $\Sigma$, which is groomed. If $\al_1$, $\al_2$ are arcs such that  $\partial \al_1\cup \partial \al_2$ intersects each component of $\partial \Sigma$ at the same sign, then $\al_1\cup \al_2$ is also groomed.

If $\al_3$ is groomed, then \eqref{eq: only one grading} and the assumption that $\dim \SHI(M,\ga,L)\le 2$ contradicts Proposition \ref{prop: two non-zero gradings}.
 
If $\al_3$ is not groomed, then both $\al_1$ and $\al_2$ are arcs. Let $\al'_1={\rm sign}(r)\cdot \al_1$, $\al'_2={\rm sign}(s)\cdot \al_2$. Then there is a unique component $\sigma$ of $\partial \Sigma$ that contains two points in $\partial \al_1\cup \partial \al_2$, and the intersections $\sigma \cap \al_1'$ and $\sigma\cap \al_2'$ have opposite signs. 

Suppose $q_1$ and $q_2$ are two intersections of $\sigma$ and $\al_3$ with opposite signs, and assume that $q_1$ and $q_2$ are adjacent on $\sigma$. Let $\delta$ be the part of $\sigma$ between $q_1$ and $q_2$ that does not contain any other intersections of $\sigma$ with $\al_3$, We can glue $\delta$ to $\al_3$ and push it into the interior of $\Sigma$ to get a new collection of arcs $\al_4$ as in Figure \ref{fig: alpha} so that
$\partial \al_3=\partial \al_4\cup\{q_1,q_2\}.$

\begin{figure}[h]
	\begin{overpic}[width=0.6\textwidth]{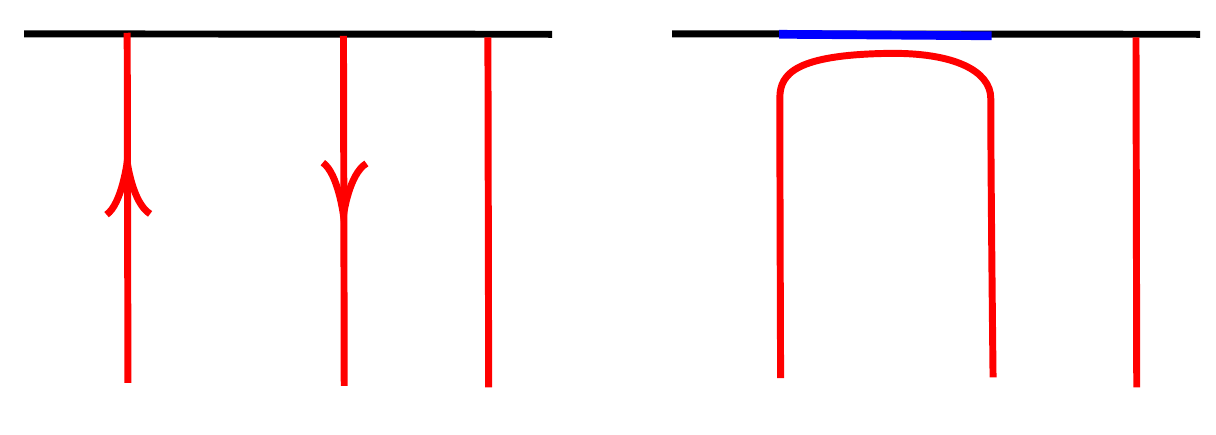}
		\put(26,1){$\al_3$}
		\put(48,32){$\partial \Sigma$}
		\put(10,36){$q_1$}
		\put(27,36){$q_2$}
		\put(63,36){$q_1$}
		\put(80,36){$q_2$}
		\put(72,35){$\delta$}
		\put(80,1){$\al_4$}
	\end{overpic}
	\vspace{\baselineskip}
	\caption{Modify $\al_3$ to obtain $\al_4$}\label{fig: alpha}
\end{figure}

Now we show that $\al_3$ and $\al_4$ induce the same grading on $\SHI(M,\ga,L)$. Pick an auxiliary surface $F$ for $(M,\ga)$ and pick a collection of arcs $\al_3'$ on $F$ so that $\partial \al_3'$ are glued to $\partial \al_3$. By Lemma \ref{lem: grading is well defined}, we can choose $\al_3'$ arbitrarily.  Choose $\al_3'$ so that one component of $\al_3'$ connects $q_1$ and $q_2$ and is boundary parallel to $\delta$ in $F$. Now  take $\al_4'=\al_3'\backslash\beta$. Using $\al_3'$ and $\al_4'$, we  extend $[-1,1]\times \al_3$ and $[-1,1]\times \al_4$ to closed surfaces $S^1\times (\al_3\cup\al_3')$ and $S^1\times (\al_4\cup\al_4')$ inside the closure $S^1\times (\Sigma\cup F)$ of $(M,\ga,L)$. It is clear that
$$[S^1\times (\al_3\cup\al_3')]=[S^1\times (\al_4\cup\al_4')]\in H_2(S^1\times (\Sigma\cup F)),$$
so the gradings associated to $\al_3$ and $\al_4$ are the same.

This process can be repeated until we obtain a groomed $1$--manifold that induces the same grading on $\SHI(M,\ga,L)$ as $\al_3$. By Lemma \ref{lem_sum_gradings_SHI}, Proposition \ref{prop: two non-zero gradings}, and (\ref{eq: only one grading}), this yields a contradiction. 
\epf

\begin{cor}\label{cor_contained_in_annulus}
\label{cor: theorem 1.3}	
Suppose $\Sigma$ is a connected compact oriented surface with non-empty boundary.
Let $(M,\ga) = ([-1,1]\times \Sigma,\{0\}\times \Sigma)$
	and suppose $\dim \SHI(M,\ga,L) \le 2$ for some link $L\subset {\rm int} (M)$. Then there exists an embedded annulus $N\subset \Sigma$ such that $L$ can be isotoped into $[-1,1]\times N$. 
\end{cor}

\begin{proof}
	The statement is obvious if $L$ is contained in an embedded $3$--ball in $M$. From now, we assume that $L$ is not contained in a $3$--ball, therefore $M$ is $L$--irreducible. Since $R(\ga)$ is incompressible and norm-minimizing in $(M,\ga)$, it is also $L$--incompressible and $L$--norm minimizing.
	Therefore $(M,\ga)$ is $L$--taut.
	
	If there is a non-trivial product annulus disjoint from $L$ or a product disk that is not boundary-parallel in $M\backslash L$, we may cut open $M$ along the annulus or disk and obtain another balanced sutured manifold $(M',\ga')$. We know that $(M',\ga')$ is also $L$--taut by \cite[Lemma 4.2]{Schar}. By Lemma \ref{lem_cutting_open_along_product_annulus_and_disk} below, $(M',\ga')$ is also a product sutured manifold. We have $$\dim \SHI(M',\ga',L) = \dim \SHI(M,\ga,L) $$
	because both sides of the equation can be defined using the same closure.

	 Repeat this process until it cannot be continued; by \cite[Lemma 2.16]{juhasz2010polytope}, this process ends after finitely many steps. By Proposition \ref{prop: two non-zero gradings} and the assumption that $\dim \SHI(M,\ga,L) \le  2$, there is at most one component in the final stage that contains a non-empty subset of $L$.
	  Let 
	  $$(\hat M,\hat \ga)= ([-1,1]\times \hat \Sigma, \{0\}\times \hat \Sigma)$$ be the connected component in the final stage that contains $L$.  By Proposition \ref{prop_SHI_rank_greater_than_2_if_no_trivial_product}, $\hat \Sigma$ has genus-zero and at most two boundary components. This implies that $\hat \Sigma$ is a disk or an annulus, therefore  the desired result is proved.
\end{proof}

\begin{lem}
	\label{lem_cutting_open_along_product_annulus_and_disk}
	Suppose $\Sigma$ is a compact surface with non-empty boundary. 
	\begin{enumerate}
		\item If $S$ is a properly embedded annulus in $[-1,1]\times \Sigma$ such that $S\cap \{\pm 1\}\times \Sigma$ is an embedded non-contractible circle on $\{\pm 1\}\times \Sigma$, then $(S,\partial S)$ is isotopic in $([-1,1]\times \Sigma, \{-1,1\}\times \Sigma)$ to an annulus of the form $[-1,1]\times \beta$, where $\beta$ is a non-contractible circle on $\Sigma$.
		\item If $S$ is a properly embedded disk in $[-1,1]\times \Sigma$ such that $S\cap \{\pm 1\}\times \Sigma$ is an embedded arc on $\{\pm 1\}\times \Sigma$, then $(S,\partial S)$ is isotopic in $([-1,1]\times \Sigma, \{-1,1\}\times \Sigma)$ to a disk of the form $[-1,1]\times \beta$, where $\beta$ is an arc on $\Sigma$.
	\end{enumerate}
\end{lem}

\begin{proof}
	For Part (1), write  $S\cap \{1\}\times \Sigma = \{1\}\times \beta_+$, $S\cap \{-1\}\times \Sigma = \{-1\}\times \beta_-$, then $\beta_+$ and $\beta_-$ are isotopic curves in $\Sigma$, and we may assume without loss of generality that $\be_-=\be_+$. 
	By the assumptions, the embedding of $S$ in $[-1,1]\times \Sigma$ is $\pi_1$--injective.
	
	Let $\ga$ be an essential arc on $\Sigma$ that is disjoint from $\be_\pm$. Perturb $S$ so that it intersects $[-1,1]\times \ga$ transversely. Then every component of $S\cap ([-1,1]\times \ga)$ is contractible in $[-1,1]\times \Sigma$, thus is contractible in $S$. A standard innermost disk argument shows that $S$ can be isotoped so that it is disjoint from $[-1,1]\times \ga$. As a result, we may cut open the surface $\Sigma$ along arcs until $\Sigma$ is an annulus and $\be_\pm$ is a non-contractible simple closed curve on $\Sigma$. 
	
	Now write $\Sigma$ as $B(2)\backslash B(1)$, where $B(r)$ is the disk of radius $r$ in $\bR^2$ centered at $(0,0)$.   Let $M_1$ and $M_2$ be the closures of the two components of $[-1,1]\times \Sigma\backslash S$, where $M_1$ contains $[-1,1]\times \partial B(1)$ and $M_2$ contains $[-1,1]\times \partial B(2)$. The Sch\"onflies theorem implies that $M_1 \cup ([-1,1]\times B(1))$ is a ball. Let $T=[-1,1]\times \{(0,0)\}$, then $T$ is a tangle in $M_1 \cup ([-1,1]\times B(1))$, and $M_1$ can be viewed as the ball $M_1 \cup ([-1,1]\times B(1))$ with a tubular neighborhood of $T$ removed.
	
	Since the map $\pi_1(S)\to \pi_1([-1,1]\times \Sigma)$ induced by inclusion is an isomorphism, it follows from the Seifert--van-Kampen theorem that the maps $\pi_1(S)\to \pi_1(M_1)$ and $\pi_1(S)\to \pi_1(M_2)$ induced by inclusions are also isomorphisms, therefore $\pi_1(M_1)\cong \pi_1(M_2) \cong \bZ$. A standard argument using Dehn's lemma then implies that $T$ is the trivial tangle in the ball $M_1 \cup ([-1,1]\times B(1))$. So $M_1$ is diffeomorphic to $[-1,1]\times A$, and $S$ can be isotoped to a parallel copy of $[-1,1]\times \partial B(1)$. Therefore Part(1) of the lemma  is proved.
	
	The proof of Part (2) follows from a similar (but simpler) argument, where we first cut $\Sigma$ open along arcs to reduce to the case where $\Sigma$ is a disk, and then apply the Sch\"onflies theorem.
\end{proof}

\begin{prop}\label{prop_AHI_rank_2}
Suppose $L$ is a link in the thickened annulus $(-1,1)\times A$. If $\dim\AHI(L)\le 2$, then $L$ is either the unknot or isotopic to a non-contractible circle in $\{0\}\times A$. 
\end{prop}
\begin{proof}
We can view the thickened annulus as the complement of an unknot $U\subset S^3$. According to \cite{AHI}*{Section 4.3},
	we have $\AHI(L)\cong \II^\natural (L\cup U;p)$, where $p\in U$ is a base point. By \cite[Proposition 3.5]{xie2021meridian}, this implies that $\dim \AHI(L) \ge 2^{|L|}$, where $|L|$ is the number of components of $L$. Since we assume that $\dim\AHI(L)\le 2$, we know that $L$ must have only one component.
	
According to \cite{Xie-earring}*{Section 3}, there is an earring-removing exact triangle
	$$
	\cdots \to \II(S^3,L\cup U \cup m_1\cup m_2, u_1\cup u_2 ) \to \II^\natural (L\cup U;p) \to \II^\natural (L\cup U;p)
	$$ 
	where $m_1, m_2$ are small meridians around $L, U$ respectively, and $u_1, u_2$ are small arcs joining $m_1,m_2$ to $L, U$ respectively. By \cite{Xie-earring}*{Proposition 5.1}, we have
	$\II(L\cup U \cup m_1\cup m_2, u_1+u_2 ) \cong \KHI(L\cup U)$. Therefore the above exact triangle implies 
	\begin{equation}
	\label{eqn_AHI_rank_2_proof1}
	\dim \KHI(L\cup U) \le 2\cdot\dim  \II^\natural (L\cup U;p) =2\cdot\dim \AHI(L) \le 4.
	\end{equation}
	
	By \cite{LY-HeegaardDiagram}*{Proposition 3.14} we have
	\begin{equation}
		\label{eqn_AHI_rank_2_proof2}
			\dim \KHI(L\cup U) \ge  2 \cdot\dim \KHI (L).
	\end{equation}

	By \cite[Theorem 1.1]{KM:Alexander}, the Euler characteristic of $\KHI$ is given by the Alexander polynomial. Since $L$ is a knot, this implies that $\dim \KHI(L)$ is odd and $\dim \KHI(L\cup U)$ is even.
	Therefore \eqref{eqn_AHI_rank_2_proof1}, \eqref{eqn_AHI_rank_2_proof2} imply that 
	$$\dim \KHI(L) = 1, \quad \dim \KHI(L\cup U) = 2 ~ \text{or} ~4.$$
	
	If $\dim \KHI(L\cup U) = 2$, then by \cite{LXZ-unknot}*{Corollary 1.8}, $L\cup U$ is the unlink, and the desired result holds. From now on, we assume that $\dim \KHI(L\cup U) = 4$.
	
	Since $\dim \KHI(L) = 1$,  \cite{KM:suture}*{Proposition 7.16} implies that $L$ is an unknot when viewed as a knot in $S^3$.  
	
	If the top f-grading of $\AHI(L)$ is zero, then by \cite{XZ:excision}*{Corollary 8.3}, $L$ is included in a solid $3$--ball in the thickened annulus, so $L\cup U$ is the unlink and $\dim \KHI(L\cup U) = 2$, contradicting the assumption that $\dim \KHI(L\cup U) = 4$. Therefore, the top f-grading of $\AHI(L)$ is not zero. By the symmetry of f-grading and the assumption that $\dim \AHI(L)\le 2$, the top f-grading component of $\AHI(L)$ must have dimension $1$. Therefore by \cite{XZ:excision}*{Corollary 8.4}, $L$ is a braid closure in the thickened annulus. Let $l>0$ be the strand number of the braid. 
	
	Let $S^3(L\cup U)$ denote the complement of a tubular neighborhood of $L\cup U$ in $S^3$, and let $T^2_L, T^2_U\subset \partial\big( S^3(L\cup U)\big)$ be the two components corresponding to $L$ and $U$ respectively. Take $\ga_{\mu}$ to be the union of two meridians of $L$ and two meridians of $U$. Then by the definition of $\KHI$ (cf. \cite{KM:suture}), we have
	$$\KHI(L\cup U)\cong \SHI(S^3(L\cup U),\ga_{\mu}).$$

	Let $D$ be an embedded disk bounded by $L$ in $S^3$. The surface $D\cap S^3(L\cup U)$ represents a homology class in $H_2(S^3(L\cup U),\partial S^3(L\cup U))$. Let $S$ be a surface representing the same homology class in $H_2(S^3(L\cup U),\partial S^3(L\cup U))$ so that $S$ is norm-minimizing, $\partial S\cap T^2_L$ is a longitude of $L$, and $\partial S\cap T^2_U$ consists of parallel and coherently oriented meridians of $U$. We can arrange the two components of $\ga_{\mu}\cap T_U$ so that they are not separated by $\partial S\cap T_U$. Note since the linking number between $L$ and $U$ is non-zero, we must have $\chi(S)\leq0$ and hence according to \cite[Corollary 6.3]{GL-decomposition}, we have two taut sutured manifold decompositions
$$
		(S^3(L\cup U),\ga_{\mu})\stackrel{S}{\leadsto}(M_+,\ga_+)~{\rm and~}(S^3(L\cup U),\ga_{\mu})\stackrel{-S}{\leadsto}(M_-,\ga_-).
$$
	
As in \cite[Section 2]{GL-decomposition}, the surface $S$ induces a $\mathbb{Z}$--grading on $\KHI(L\cup U)=\SHI(S^3(L\cup U),\ga_{\mu})$ after possible stabilizations (cf. \cite[Definition 2.23]{GL-decomposition}). Now since $L$ and $U$ have a non-zero linking number, \cite[Theorem 1.18]{GL-decomposition} applies and it states that $\KHI(L\cup U)$ detects the Thurston norm of the homology class $[S]\in H_2(S^3(L\cup U),\partial S^3(L\cup U))$. In \cite[Section 6]{GL-decomposition}, this result is proved by establishing the following statements:
\begin{equation}\label{eq: KHI(L cup U)}
	\KHI(L\cup U,i_{t})\cong \SHI(M_+,\ga_+),~\KHI(L\cup U,i_{b})\cong \SHI(M_-,\ga_-),
\end{equation}
and
\begin{equation}\label{eq: i_t-i_b}
	x([S])=-\chi(S)=i_t-i_b-1,
\end{equation}
where $i_t$ and $i_b$ are the top and bottom gradings of $\KHI(L\cup U)$ with respect to the grading associated with $S$, and $\KHI(L\cup U,i)$ denotes the component of $\KHI(L\cup U)$ that has grading $i$ with respect to the grading induced by $S$.

Now we look at the balanced manifolds $(M_+,\ga_+)$ and $(M_-,\ga_-)$. Since we have chosen the suture $\ga_\mu$ in a way that the two components of $\ga_{\mu}\cap T^2_U$ are not separated by $\partial S\cap T^2_U$, there are three components of $\ga_+$ that are parallel to each other: one comes from $\partial S\cap T^2_U$, and two come from $\ga_{\mu}\cap T^2_U$. The same statement holds for $\ga_-$. Let $\ga_{\pm}'$ be the suture $\ga_{\pm}$ with two of the three parallel components removed. From \cite[Proof of Theorem 3.1]{KM:Alexander}, we know that
$$\SHI(M_{\pm},\ga_{\pm})\cong\SHI(M_{\pm},\ga_{\pm}')\otimes\mathbb{C}^2.$$
Note Equation (\ref{eq: i_t-i_b}) implies that $i_t\neq i_b$, and hence from Equation (\ref{eq: KHI(L cup U)})we know
$$\bigg(\SHI(M_+,\ga_{+}')\otimes\mathbb{C}^2\bigg)\oplus\bigg(\SHI(M_-,\ga_{-}')\otimes\mathbb{C}^2\bigg)\hookrightarrow\KHI(L\cup U)\cong\mathbb{C}^4.$$
As a result we must have $\SHI(M_{\pm},\ga_{\pm}')\cong\mathbb{C}$. By \cite[Theorem 1.2]{GL-decomposition} (see also \cite{KM:suture}*{Theorem 6.1} for the case then $(M_\pm, \ga_\pm')$ is a homology product), $(M_{\pm},\ga_{\pm}')$ are both product sutured manifolds, and hence $U$ is the closure of a braid with axis $L$. 

Recall that we also proved $L$ is a braid closure in the thickened annulus with strand number $l>0$. So $L\cup U$ is a mutually braided link. 
According to \cite{XZ:forest}*{Lemma 6.7}, if $|l|\ge 2$, then 
$$
\| (1-x)(1-y)\tilde{\Delta}_{L\cup U}(x,y)\|>4,
$$
where $\tilde{\Delta}_{L\cup U}(x,y)$ denotes the multi-variable Alexander polynomial of $L\cup U$ and $\|\cdot \|$ denotes the sum of the absolute values of the coefficients of 
a polynomial.

By \cite{LY_multi_Alexander}*{Theorem 1.4}, $\KHI(L\cup U)$ can be equipped with multi-gradings whose graded Euler characteristic recovers 
$(1-x)(1-y)\tilde{\Delta}_{L\cup U}(x,y)$. Therefore when $l>2$, we have
$$
4=\dim \KHI(L\cup U) \ge \| (1-x)(1-y)\tilde{\Delta}_{L\cup U}(x,y)\|>4,
$$
 which yields a contradiction.
As a consequence, we must have $l= 1$, so $L\cup U$ is the Hopf link, and the desired result is proved.
\end{proof}

\begin{proof}[Proof of Theorem \ref{thm_main_instanton}]
Suppose
	$$
	\dim_{\bC}\SHI([-1,1]\times \Sigma, \{0\}\times \Sigma, L)\le 2.
	$$
	By Corollary \ref{cor_contained_in_annulus}, this implies that there exists an embedding of the annulus $\varphi:A\to \Sigma$ such that $L$ is contained in $(-1,1)\times \varphi(A)$ after isotopy. If $\varphi$ is contractible, we may find a non-contractible annulus embedded in $\Sigma$ that contains $\varphi(A)$, so we may assume that $\varphi$ is not contractible without loss of generality.
	
	Let $\varphi^{-1}(L)$ be the link in the interior of $[-1,1]\times A$ given by the pull-back of $L$ by $\id_{[-1,1]}\times \varphi$. 
By \cite[Lemma 3.10]{xie2020instantons} and Proposition \ref{prop_SiHI_iso_SHI},
	$$
	\AHI(\varphi^{-1}(L))\cong \SHI([-1,1]\times \Sigma, \{0\}\times \Sigma, L).
	$$
 By Proposition \ref{prop_AHI_rank_2}, we conclude that $\varphi^{-1}(L)$ is isotopic to a knot embedded in $\{0\}\times A$ via an isotopy in $(-1,1)\times A$. Therefore the theorem is proved.
\end{proof}

\subsection{Proof of Lemma \ref{lem: juhasz's ineq}}\label{subsec: key inequality}
This subsection proves the key inequality Lemma \ref{lem: juhasz's ineq}, thus completing the proof of Proposition \ref{prop: two non-zero gradings} and Theorem \ref{thm_main_instanton}. Our argument follows the strategy of the proof of \cite[Theorem 6.1]{juhasz2010polytope}. 

\blem\label{lem: juhasz's ineq}
Suppose $(M,\ga)$ is a taut balanced sutured manifold and $L\subset {\rm int}(M)$ is a link so that $(M,\ga)$ is $L$--taut and has no non-trivial horizontal surfaces (cf. \cite[Definition 2.17]{juhasz2010polytope}), non-trivial product annuli, or non-boundary-parallel product disks that are disjoint from $L$. Assume further that $H_2(M;\bZ)=0$. Suppose $S_+$ and $S_-$ are two surfaces so that the following hold:
\begin{itemize}
	\item [(i)] $S_+\cap R(\ga)=-S_-\cap R(\ga)$.
	\item [(ii)] 
	$[S_+,\partial S_+]=-[S_-,\partial S_-]\neq0\in H_2(M,\partial M)$
	\item [(iii)] There is a positive integer $k$ so that 
	$$[S_+,\partial S_+]+[S_-,\partial S_-]=k\cdot [R_+(\ga),\partial R_+(\ga)]\in H_2(M,N(\ga)).$$
	\item [(iv)] The surfaces $S_{\pm}$ are both $x_L$-norm minimizing and $L$--incompressible (cf. \cite[Definition 1.2]{Schar}).
\end{itemize}
Then we have the following inequality:
$$\chi(S_+)+\chi(S_-)-|S_+\cap \ga|-|L\cap S_+|-|L\cap S_-|<k\cdot \chi(R_+(\ga)).$$
\elem
\bpf
Let $P$ be the result of the cut and paste surgery of $S_+\cup S_-$ described as follows: on $R(\ga)$, Condition (i) above states that $\partial S_+\cap R(\ga)=\partial S_-\cap R(\ga)$, and we  glue $\partial S_+\cap R(\ga)$ to $\partial S_-\cap R(\ga)$; away from $R(\ga)$, perturb $S_+$ and $S_-$ so that they intersect transversely, cut them open along intersections, and re-glue in an orientation preserving way to resolve the intersections.

By the first inequality on \cite[Page 27]{juhasz2010polytope},  we have
$$\chi(S_+) + \chi(S_-) - |S_+\cap \ga| \le \chi (P).$$
Hence it suffices to show that
$$\chi(P)-|L\cap P|<k\cdot \chi(R_+(\ga)).$$
To do this, first note that
$$\partial[P]=\partial[S_+]+\partial[S_-]=\partial[S_+]+\partial[-S_+]+k\partial[R_+(\ga)]=k\cdot[\ga]\in H_1(A(\ga);\bZ).$$
Since $H_2(M)=0$, we know from Condition (iii) that
$$[P]=k\cdot[R_+(\ga)]\in H_2(M,A(\ga);\bZ).$$

From Condition (iv), we know that any closed component of $S_+\cap S_-$ are either both essential or both inessential on $S_+$ and $S_-$. 

We may assume without loss of generality that the intersection points of $S_+$ with every component of $L$ have  the same sign. In fact, if a component of $L$ intersects $S_+$ at two consecutive points with opposite signs, we may attach a tube along the segment of $L$ between these two points to $S_+$ without changing the homology class of $S_+$ or the Thurston norms.  The same statement holds for $S_-$. 

If $P$ contains a spherical component $Q$ that intersects $L$ transversely at no more than one point, since $H_2(M;\bZ)=0$, the component $Q$ must be disjoint from $L$. By a standard innermost circle argument, we can isotope $S_+$ and $S_-$ to eliminate this $2$--spherical component of $P$ without changing the intersection numbers with $L$. 

If $P$ contains a toroidal component $T$ that is disjoint from $L$, 
then after isotopy of $S_+$ and $S_-$ using a standard innermost circle argument, we can assume that $T$ consists of some annuli on $S_+$ and some annuli on $S_-$. Since $H_2(M;\bZ)=0$ any torus inside $M$ is homologically trivial. As a result, we can perform cut and paste surgeries to $S_+$ and $S_-$ to swap those annular parts on $T$. After the cut and paste surgeries, $S_{\pm}$ remain in the same homology classes and have same generalized Thurston norms as before, but the closed component $T$ is merged into other components of $P$.

If $P$ contains a spherical component $Q$ that intersects $L$ transversely at two points, since $H_2(M;\bZ)=0$, the algebraic intersection number of $Q$ and $L$ is zero. 
Since we assume that the intersection points of $S_\pm$ and every component of $L$ have the same sign, one of the intersection points in $Q\cap L$ is contained in $S_+$, and the other is contained in $S_-$. 
After isotoping $S_+$ and $S_-$ using the innermost circle argument, we may assume that the decomposition of $Q$ by $S_+$ and $S_-$ consists of an even number of annuli and two disks. As a result, we can perform cut and paste surgeries to $S_+$ and $S_-$ to swap those annular parts on $Q$. After the cut and paste surgeries, $S_{\pm}$ remain in the same homology classes and have same generalized Thurston norms as before, but the closed component $Q$ is merged into other components of $P$.

If $P$ has a disk component, then by the definition of $P$, $\partial P$ is contained in a small tubular neighborhood of $\ga$, so $\partial D$ is parallel to a component of $\ga$. Hence $R_{\pm}(\ga)$ must both be disks as they are incompressible. Since $(M,\ga)$ is assumed to be taut, $M$ must be a $3$-ball, but then $(M,\ga)$ cannot be $L$--taut. 

Hence, we may assume without loss of generality that every component of $P$ has a positive $L$--Thurston norm. So
$$x_L(P)=|L\cap P|-\chi(P)~{\rm and~}x_L(R_+(\ga))=x(R_+(\ga))=-\chi(R_+(\ga)).$$
The desired inequality is then equivalent to
$$x_L(P)>k\cdot x(R_+(\ga)).$$

Assume the contrary; namely,
$x_L(P)\leq k\cdot x(R_+(\ga)).$
Since we know that
$$[P]=k\cdot [R_+(\ga)]\in H_2(M,A(\ga);\bZ)$$
and $R_+(\ga)$ is norm-minimizing, we have
$$k\cdot x(R_+(\ga))\leq x(P)\leq x_L(P)\leq k\cdot x(R_+(\ga)).$$
Since every component of $P$ has positive generalized Thurston norm and $H_2(M;\bZ)=0$, we conclude that $P$ has no closed components, $P\cap L= \emptyset$, and 
\begin{equation}
	\label{eqn_norm_of_P}
x(P)=k\cdot x(R_+(\ga)).
\end{equation}

 Fix an arbitrary base point $x\in M\backslash P$ and define a function $\phi:M\backslash P\to \mathbb{Z}$ by taking $\phi(y)$ to be the algebraic intersection number of any path from $x$ to $y$ with $P$. Since $[P]=k\cdot [R_+(\ga)]=0\in H_2(M,\partial M)$, the map $\phi$ is well defined. Since $P$ has no closed components, this implies that $P$ has the form
$$P= P_1\cup...\cup P_{k},$$
where each $\partial P_i$ is parallel to $\ga$.

Since $M\backslash L$ does not contain non-trivial horizontal surfaces, we know that each $P_i$ is either parallel to $R_+(\ga)$ or parallel to $R_-(\ga)$ in $M\backslash L$. We can assume that $P_i$ is parallel to $R_+(\ga)$ for $1\leq i\leq j$ and $P_i$ is parallel to $R_-(\ga)$ for $j+1\leq i\leq k$. Let $J_j$ be the part of $M$ between $P_{j}$ and $P_{j+1}$, and assume without loss of generality that $\ga\subset \partial J_j\cap A(\ga)$. Then $(J_j,\ga)$ is balanced sutured manifold diffeomorphic to $(M,\ga)$ and $L\subset {\rm int}(J_j)$. 

Observe that the difference between $S_+\cup S_-$ and $P$ in the interior of $M$ are within a neighborhood of $S_+\cap S_-$ along which we performed the cut and paste surgery. As a result, every component of $S_+\cap J_j$ is either a product disk or a product annulus in $J_j$; product disk components of $S_+\cap J_j$ correspond to arc components of $S_+\cap S_-\cap {\rm int}(M)$, and product annulus components of $S_+\cap J_j$ correspond to circle components of $S_+\cap S_-\cap {\rm int}(M)$. Also note that these product disks and product annuli are disjoint from $L$. So by the assumption on $(M,\ga,L)$, we know that $S_+\cap J_j$ can be isotoped into $\partial J_j$. Since $M\backslash J_j$ is a product, we can deform $S_+$ into $\partial M$ in $M$. This implies 
$$[\partial S_+]=0\in H_1(\partial M;\bZ).$$
Since $H_2(M;\bZ)=0$, it contradicts Condition (ii).
\epf

\section{Instanton homology for links in thickened surfaces}
\label{sec_instanton_thickend_surface}

Let $R$ be a closed, connected, oriented surface. 
In this section, we define and study the properties of an instanton Floer homology invariant for links in $(-1,1)\times R$.

Recall that $t_*\in S^1$ denotes the image of $\{-1,1\}\subset[-1,1]$ under the quotient map.

\begin{defn}
	Let $L$ be a link in $(-1,1)\times R$. Let $p$ be a point on $R$ that is disjoint from the projection of $L$ to $R$.  
	Embed $L$ into $S^1\times R$ via the quotient map $[-1,1] \times R \to S^1\times R$. 
	Define
	$$
	\SiHI_{R,p}(L) = \II(S^1\times R, L, S^1\times \{p\}|\{t_*\}\times R).
	$$
\end{defn}

\begin{rem}
	The group $\SiHI_{R,p}(L)$ is the same as the group $\SiHI(L)$ defined in \cite{xie2020instantons}. We need to study the properties of $\SiHI_{R,p}(L)$ with different surfaces $R$ in this paper, so we modify the notation to keep track of $R$ and $p$. Since the homology class $[S^1\times{p}]\in H_2(S^1\times R,L)$ is independent of $p$, the isomorphism class of $\SiHI_{R,p}(L)$ is independent of $p$.
\end{rem}

Suppose $\Sigma$ is a compact connected oriented surface with $\partial \Sigma\neq \emptyset$, and $L$ is a link in the interior of $[-1,1]\times \Sigma$. 
	Let $F$ be a connected oriented surface such that there is an orientation-reversing diffeomorphism $\tau: \partial F\to\partial \Sigma$, and let $R = \Sigma\cup_\tau F$. We also assume that $F$ is not a disk.
	Then by Proposition \ref{prop_SiHI_iso_SHI}, we have
\begin{equation}
	\label{eqn_SiHI_iso_SHI}
	\SiHI_{R,p}(L) \cong \SHI([-1,1]\times \Sigma, \{0\}\times \partial \Sigma, L).
\end{equation}

Suppose $S$ is a link cobordism from $L_1$ to $L_2$ in the interior of $[-1,1]\times [-1,1]\times R$ whose projection to $R$ is disjoint from $p$. Then 
$$([-1,1]\times S^1\times R, S, [-1,1]\times S^1\times \{p\})$$
is a cobordism from $(S^1\times R,L_1,S^1\times \{p\})$ to $(S^1\times R,L_2,S^1\times \{p\})$. We will denote the induced cobordism map by $\SiHI_{R,p}(S)$.

\subsection{Embedding of annulus in $R$}

Let $\varphi:A\to R$ be an orientation-preserving embedding of the annulus $A$ in $R$ such that the image of $A$ is non-separating in $R$. Suppose $L$ is a link in the interior of $[-1,1]\times A$. We will abuse notation and use $L$ to denote its image in $[-1,1]\times R$. Let $p$ be a point in $R\backslash \varphi(A)$. 

The main result of this subsection is the following.

\begin{prop}
	\label{prop_embedding_of_annulus_in_R}
	There exists a natural isomorphism from $\AHI(L)$ to $\SiHI_{R,p}(L)$.
\end{prop} 

By definition, an isomorphism $\mathcal{I}: \AHI(L)\to \SiHI_{R,p}(L)$ is natural if it intertwines with link cobordism maps. More precisely, for every link cobordism $S$ from $L_1$ to $L_2$ in $[-1,1]\times [-1,1]\times A$, we require that
$$
\mathcal{I}\circ \AHI(S) = \pm \SiHI_{R,p}(S)\circ \mathcal{I}.
$$
\begin{proof}
	Let $\varphi':A\to T^2$ be an embedding of $A$ in the torus such that the image is parallel to a non-trivial simple closed curve in $T^2$, and let $F'$ be the closure of $T^2\backslash \varphi'(A)$. Let $p'$ be a point on $F'$.  By \cite[Lemma 3.10]{xie2020instantons}, $\AHI(L)\cong \SiHI_{T^2,p'}(L)$.
	
	By \eqref{eqn_SiHI_iso_SHI}, $\SiHI_{T^2,p'}(L)\cong \SHI([-1,1]\times A, \{0\}\times A, L) \cong \SiHI_{R,p}(L)$. 
	Therefore $\AHI(L)\cong \SiHI_{R,p}(L)$.
	
	All the isomorphisms above are given by excisions along surfaces disjoint from $(-1,1)\times A$, so the excision surfaces are disjoint from links and link cobordisms in $(-1,1)\times A$, therefore all the isomorphisms above intertwine with link cobordism maps. 
\end{proof}

\begin{cor}
	\label{cor_zero_maps_detected_by_AHI}
	Suppose $L_1,L_2$ are links in $(-1,1)\times A$ and $S$ is a cobordism from $L_1$ to $L_2$. Let $\varphi$ be an orientation-preserving embedding of $A$ in $R$. Let $p$ be a point in $R\backslash \varphi(A)$.  Then
	\begin{enumerate}
		\item A cobordism $S$ from $L_1$ to $L_2$ defines the zero map on $\SiHI_{R,p}(L_1)$ if and only if $S$ defines the zero map on $\AHI(L_1)$. 
		\item 	Suppose $S_1,\dots,S_k$ are $k$ cobordisms from $L_1$ to $L_2$, and $\lambda_1,\dots,\lambda_k\in \bC$. Then $\sum \lambda_i\AHI(S_i)=0$ for some choice of signs of the cobordism maps if and only if $\sum \lambda_i\SiHI_{R,p}(S_i)=0$ for some choice of signs of the cobordism maps. 	
	\end{enumerate}
\end{cor}
 
\subsection{Grading on $\SiHI$ from simple closed curves}
 Let $c$ be an oriented simple closed curve in $R$. The following result was proved in \cite{xie2020instantons}.

\begin{lem}[{\cite[Lemma 3.16]{xie2020instantons}}]
	\label{lem_grading_SiHI_range}
	Suppose $S^1\times c$ intersects $L$ transversely at $n$ points, then the eigenvalues of $\muu(S^1\times c)$ on $\SiHI_{R,p}(L)$ are contained in the set 
	$\{-n+2i|i\in\bZ, 0\le i\le n\}$.
\end{lem}

Lemma \ref{lem_grading_SiHI_range} is the motivation of the following definition.

\begin{defn}
	Define the $c$--grading on $\SiHI_{R,p}(L)$ to be the grading given by the generalized eigenspace decomposition of $\muu(S^1\times c)$. 
\end{defn}

Suppose $S$ is a link cobordism.
By Corollary \ref{cor_mu_map_eigenspace_preserved_by_cobordism}, for every oriented simple closed curve $c$ on $R$, the map $\SiHI_{R,p}(S)$ preserves the $c$--grading.

\subsection{Maps induced by diffeomorphisms}
Recall from Example \ref{exmp_diff_induce_iso} that diffeomorphisms on admissible triples define isomorphisms on instanton homology.
In this subsection, we prove the following result for maps on $\SiHI_{R,p}$ induced by diffeomorphisms. 

\begin{prop}
	\label{prop_diff_induces_id_on_II}
	Suppose $R$ is a closed oriented surface. Suppose $N\subset R$ is a connected embedded compact surface with boundary, such that $\partial N$ has an even number of connected components. Assume that $R\backslash N$ is connected and has genus at least $3$. Let $L$ be a link in the interior of $(-1,1)\times N$. Let $p\in R$ be a point disjoint from the projection of $L$ to $R$. Let $\varphi:R\to R$ be a diffeomorphism that restricts to the identity on $N\cup \{p\}$. Then the diffeomorphism $\id_{S^1}\times \varphi$ induces $\pm \id$ on $\SiHI_{R,p}(L)$.
\end{prop}

We start by introducing the following notation.
If $\omega$ is a closed $1$--manifold embedded in $R$, $L$ is a link in $(-1,1)\times R$, and $p\in R\backslash \omega$ is a point disjoint from the projection of $L$ to $R$, define
 $$
H(L,\omega,p)= \II(S^1\times R, L,\{t_*\}\times \omega \cup S^1\times \{p\}|\{t_*\}\times R).
 $$
By Proposition \ref{prop_product_space_homology_rank_1}, $H(\emptyset,\omega,p)\cong \bC$. 
If $\varphi:R\to R$ is an orientation-preserving diffeomorphism that restricts to the identity on the union of $\omega$, $\{p\}$, and the projection of $L$ to $R$, we will use $\varphi_*$ to denote the automorphism of $H(L,\omega,p)$ induced by $\id_{S^1}\times \varphi$. 

\begin{lem}
	\label{lem_diff_induce_pm_id_with_empty_link}
	Suppose $\varphi:R\to R$ is an orientation-preserving diffeomorphism such that $\varphi(p)=p$ and $\varphi(\omega)=\omega$. Then $\varphi_*=  \pm \id$ on $H(\emptyset,\omega,p)$. 
\end{lem}

\begin{proof}
	We first consider the case when $\omega=\emptyset$.
	
	Let $P$ be the pair-of-pants cobordism from two circles to one circle.  Topologically, $P$ is a sphere with $3$ interior disks removed. Taking the product of $P$ with $R$ gives a cobordism from $S^1\times R \sqcup S^1 \times R$ to $S^1\times R$. This cobordism defines a map 
	$$\mathcal{P}_0: H(\emptyset,\emptyset,p)\otimes H(\emptyset,\emptyset,p)\to H(\emptyset,\emptyset,p).$$
	By Proposition \ref{prop_excision}, $\clP_0$ is an isomorphism. The map $\mathcal{P}_0$  can be defined without sign ambiguity because the singular set is empty and the cobordism has an almost complex structure defined by taking the product of an almost complex structure on $P$ with an almost complex structure on $R$. 
	Since the singular set is empty, the map $\varphi_*$ is also defined without sign ambiguity. And we have 
	\begin{equation}
		\label{eqn_clP_varphi_commutative_H(empty,empty,p)}
	\clP_0(\varphi_*(x)\otimes \varphi_*(y)) = \varphi_*(\clP_0(x,y))
	\end{equation}
for all $x,y\in H(\emptyset,\emptyset,p)$
	because the map $\varphi$ lifts to a diffeomorphism $\id_P\times \varphi$ on the cobordism $P\times R$, and the almost complex structure is preserved up to isotopy under this diffeomorphism.
	 
	Since $H(\emptyset,\emptyset,p)\cong\bC$, there is a unique element $u_0\in H(\emptyset,\emptyset,p)$ so that $\clP(u_0\otimes u_0)=u_0.$ By \eqref{eqn_clP_varphi_commutative_H(empty,empty,p)}, we have $\varphi_*(u_0) = u_0$, so $\varphi_*=\id$ on $H(\emptyset,\emptyset,p)$.
	
	When $\omega\neq 0$, the pair-of-pants cobordism defines a homomorphism
		$$\clP_\omega: H(\emptyset,\omega,p)\otimes H(\emptyset,\omega,p)\to H(\emptyset,\emptyset,p).$$
	The same almost complex structure as above defines a sign for $\clP_\omega$, and we have 
	$$\clP_\omega(\varphi_*(x)\otimes \varphi_*(y)) = \varphi_*(\clP_\omega(x,y))$$
	for all $x,y\in H(\emptyset,\omega,p)$. Since $H(\emptyset,\omega,p)\cong\bC$, there exist exactly two elements $u_\omega, -u_\omega \in H(\emptyset,\omega,p)$ such that $\clP_\omega(u_\omega\otimes u_\omega) = u_0$. Therefore $\varphi_*(u_\omega) = \pm u_\omega$, and the lemma is proved.
\end{proof}

Now we prove a special case of Proposition \ref{prop_diff_induces_id_on_II} when $\varphi$ is given by a Dehn twist along a non-separating curve in $R\backslash N$.
\begin{lem}
	\label{lem_diff_induces_id_on_II_Dehn_twist}
Let $R, N, p$ be as in Proposition \ref{prop_diff_induces_id_on_II}.
Suppose $\ga$ is an oriented simple closed curve on $R\backslash N$ such that there exists a simple closed curve $\ga'$ on $R\backslash N$ intersecting $\ga$ transversely at one point. Let $D_\ga$ be the Dehn twist along $\ga$. Then Proposition \ref{prop_diff_induces_id_on_II} holds for $\varphi = D_\ga$. 
\end{lem}

\begin{proof}
Let $\omega$ be a simple closed curve on $R$ that is disjoint from $p$, $\ga$, $\ga'$ and intersects every component of $\partial N$ transversely at one point. Such a curve exists because we assume both $N$ and $F$ are connected and $\partial N$ has an even number of components.
Then we have the following commutative diagram (up to sign)
	$$
	  \begin{tikzcd}
		H(\emptyset,\omega,p)\otimes H(L,\omega,p) \arrow{r}\arrow{d}{\id\otimes (D_\ga)_*}  & H(L,\omega,p)\otimes H(\emptyset,\omega,p)  \arrow{d}{\id\otimes (D_\ga)_*} \\
		H(\emptyset,\omega,p)\otimes H(L,\omega,p) \arrow{r}  & H(L,\omega,p)\otimes H(\emptyset,\omega,p) ,
	\end{tikzcd}
$$
where the horizontal maps are given by excisions along the components of $S^1\times \partial N$ between two copies of $S^1\times R$. 
This diagram is commutative because the excision locus are pointwise fixed under the map $\id_{S^1}\times D_\ga$, so the cobordisms defining the two compositions are diffeomorphic.
All maps in the above diagram are isomorphisms, the two horizontal maps are identical, and by Lemma \ref{lem_diff_induce_pm_id_with_empty_link}, the right vertical map is $\pm \id$. Therefore the left vertical map is also $\pm \id$, which implies $(D_\ga)_* = \pm \id$ on $H(L,\omega,p)$.

Now consider the commutative diagram (up to sign)
$$
\begin{tikzcd}
	H(\emptyset,\omega,p)\otimes H(L,\emptyset,p) \arrow{r}\arrow{d}{(D_\ga)_*\otimes (D_\ga)_*}  & H(L,\omega,p)  \arrow{d}{(D_\ga)_*} \\
	H(\emptyset,\omega,p)\otimes H(L,\emptyset,p) \arrow{r}  & H(L,\omega,p),
\end{tikzcd}
$$
where the horizontal maps are given by excisions along $\{t_*\}\times R$. This diagram is commutative because  $\id_{S^1}\times D_\ga:S^1\times R \to S^1\times R$ lifts to a self-diffeomorphism on the cobordism that defines the horizontal maps.
Since $(D_\ga)_*=\pm  \id$ on $H(\emptyset,\omega,p)$ and $H(L,\omega,p)$, we conclude that $(D_\ga)_*= \pm \id$ on $ H(L,\emptyset,p)$.
\end{proof}

\begin{proof}[Proof of Proposition \ref{prop_diff_induces_id_on_II}]
	Since the mapping class group is generated by Dehn twists, we only need to prove Proposition \ref{prop_diff_induces_id_on_II} when $\varphi$ is a Dehn twist along a simple closed curve $\ga$ in $R\backslash (N\cup\{p\})$.
	
	Let $F=R\backslash N$. 
	If $\ga$ is non-separating in $F$, the desired result follows from Lemma \ref{lem_diff_induces_id_on_II_Dehn_twist}.
	If $\ga$ is separating, at least one component of $F\backslash \ga$ has genus at least $2$. 
	 Hence there exists an embedded compact surface $K$ in $F$ such that
	\begin{enumerate}
		\item $K$ is diffeomorphic to a sphere with four interior disks removed.
		\item $\partial K = \ga\cup\ga_1\cup \ga_2\cup \ga_3$ where $\ga_1,\ga_2,\ga_3$ are non-separating in $F$.
	\end{enumerate}
The lantern relation (see, for example, \cite[Proposition 5.1]{farb2011primer}) on $K$ then implies that the Dehn twist on $\ga$ can be generated by Dehn twists along non-separating curves on $F$. Hence the proposition is proved.
\end{proof}

\subsection{Pair-of-pants cobordisms}
\label{subsec_pair_of_pants_cobordism}
 
Suppose $L_1$ and $L_2$ are two links in $(-1,1)\times R$. Define $L_1\sqcup L_2$ to be the link obtained by isotoping $L_1$ to $(-1,0)\times R$ via the linear isotopy in the $(-1,1)$--component, isotoping $L_2$ to $(0,1)\times R$ via the linear isotopy in the $(-1,1)$--component, and then taking the union. Note that $L_1\sqcup L_2$ may not be isotopic to $L_2\sqcup L_1$ in general.

The disjoint union operation is associative: the link $(L_1\sqcup L_2) \sqcup L_3$ and $L_1\sqcup (L_2\sqcup L_3)$ are canonically isotopic by a linear isotopy in the $(-1,1)$ component. 
We use this canonical isotopy to identify $(L_1\sqcup L_2) \sqcup L_3$ and $L_1\sqcup (L_2\sqcup L_3)$, and we will denote this link by $L_1\sqcup L_2\sqcup L_3$. More generally, suppose $L_1,\dots,L_k$ are $k$ links in $(-1,1)\times R$, we define the disjoint union inductively by  setting 
$$L_1\sqcup \dots \sqcup L_k = (L_1\sqcup \dots \sqcup L_{k-1})\sqcup L_k.$$

Excision along $\{t_*\}\times R$ gives an isomorphism 
$$
\mathcal{P}:  \SiHI_{R,p}(L_1)\otimes \SiHI_{R,p}(L_2)\to \SiHI_{R,p}(L_1\sqcup L_2).
$$ 
The map $\mathcal{P}$ is defined up to sign. It is natural and associative in the following sense:
\begin{lem}
	\label{lem_associativity_naturality_clP}
	\begin{enumerate}
	\item Suppose $S_1$ is a link cobordism from $L_1$ to $L_1'$. Let $S_1'$ be the cobordism from $L_1\sqcup L_2$ to $L_1'\sqcup L_2$  defined by taking the disjoint union of $S_1$ with the product cobordism of $L_2$. Then 
	$$\SiHI_{R,p}(S_1')\circ \mathcal{P}= \pm \mathcal{P}\circ \big(\SiHI_{R,p}(S_1)\otimes \id \big). $$
	The analogous statement also holds for cobordisms on $L_2$.
	\item $\mathcal{P}(\mathcal{P}(x,y),z)= \pm  \mathcal{P}(x,\mathcal{P}(y,z))$.
	\end{enumerate}
\end{lem}

\begin{proof}
	The cobordisms defining both sides of equations are diffeomorphic.
\end{proof}

\begin{defn}
	\label{defn_clP}
	Define the isomorphism
$$
\mathcal{P}:  \SiHI_{R,p}(L_1)\otimes \SiHI_{R,p}(L_2)\otimes \dots \otimes  \SiHI_{R,p}(L_k)\to \SiHI_{R,p}(L_1\sqcup \dots\sqcup L_k).
$$
inductively by setting
$$
\mathcal{P}(x_1\otimes\dots\otimes x_k) = \clP(\mathcal{P}(x_1\otimes\dots\otimes x_{k-1}) \otimes x_k).
$$
\end{defn}

The links $L\sqcup \emptyset$ and $\emptyset\sqcup L$ are canonically isotopic to $L$ via linear isotopies on the $(-1,1)$ factor. Using these isotopies to identify $\SiHI_{R,p}(L)$ with $\SiHI_{R,p}(L\sqcup \emptyset)$ and $\SiHI_{R,p}( \emptyset\sqcup L)$, we have the following result.

\begin{lem}
	\label{lem_multiplicative_unit}
	There exists an element $ u_0\in \SiHI_{R,p}(\emptyset)$ such that $\mathcal{P}( u_0, x) = \pm x$, $\mathcal{P}( x, u_0) = \pm x$, for all links $L$ disjoint from $\{p\}\times S^1$ and all $ x\in \SiHI_{R,p}(L)$.
\end{lem}

\begin{proof}
	Let $u_0\in \SiHI_{R,p}(\emptyset)$ be an element such that $\mathcal{P}( u_0, u_0) = \pm u_0$. Then the desired result follows by the associativity of $\mathcal{P}$ and the fact that $\mathcal{P}( u_0,-)$ and $\mathcal{P}(-, u_0)$ are both surjective.
\end{proof}

\section{Instanton homology for links in $\{0\}\times R$}
\label{sec_instanton_for_links_at_0}
\subsection{Notation}
Suppose $R$ is a closed oriented surface with genus at least $3$. 
Let $p\in R$ be a fixed point.
Suppose $C\subset R\backslash\{p\}$ is a closed $1$--manifold such that every non-contractible component is non-separating.
Define a complex linear space $V_\bC(C)$ as follows:
\begin{enumerate}
	\item  If $\ga$ is a contractible simple closed curve on $R$, define $V_\bC(\ga)$ to be the space generated by $\bfv_+(\ga)$ and $\bfv_-(\ga)$, where $\bfv_+(\ga)$ and $\bfv_-(\ga)$ are formal generators associated with $\ga$. 
	\item If $\ga$ is a non-contractible simple closed curve, let $\mathfrak{o}$, $\mathfrak{o}'$ be the two orientations of $\ga$. Define $V_\bC(\ga)$ to be the space generated by $\bfw_{\mfo}(\ga)$ and $\bfw_{\mfo'}(\ga)$, where $\bfw_{\mfo}(\ga)$ and $\bfw_{\mfo'}(\ga)$ are formal generators.
	\item  In general, suppose the connected components of $C$ are $\ga_1,\dots,\ga_k$, define 
	$$V_\bC(C)=\bigotimes_{i=1}^k V_\bC(\ga_i)$$
	where the tensor product is taken with coefficient $\bC$.
\end{enumerate}
\begin{rem}
Suppose $\Sigma$ is a surface with genus zero and there is an embedding $\iota: \Sigma\to R$ such that $\iota$ induces an injection on $\pi_1$. Then the space $V_\bC(C)$ above coincides with $V(C)\otimes_\bZ \bC$, where $V(C)$ is the free group defined in Section \ref{subsec_APS}.
\end{rem}

Recall that we view $S^1$ as the quotient space of $[-1,1]$.  We will use $0\in S^1$ to denote the image of $0\in[-1,1]$ in $S^1$.
For simplicity, when $R$ and $p$ are clear from context, we will write $\SiHI_{R,p}(\{0\}\times C)$ as $\SiHI(C)$.

\subsection{Isomorphism from $V_\bC(C)$ to $\SiHI(C)$}
\label{subsec_iso_V_to_SiHI}
Let $R,p$ be as above.

\begin{lem}
	\label{lem_SiHI(C)_iso_V(C)}
	Suppose $C$  is a closed $1$--manifold in $R\backslash\{p\}$ such that every non-contractible component is non-separating. Then 
	$$\dim_\bC \SiHI(C) \cong V_{\bC}(C).$$
\end{lem}

\begin{proof}
	Suppose $C$ has $k$ components $\ga_1,\dots,\ga_k$. Then $C$ is isotopic to 
	$$(\{0\}\times \ga_1)\sqcup(\{0\}\times \ga_2)\sqcup  \dots\sqcup (\{0\}\times \ga_k).$$
	By the excision theorem (see also Section \ref{subsec_pair_of_pants_cobordism}), this implies 
	$\SiHI(C)\cong \otimes_{i=1}^k \SiHI(\ga_i).$
	Since every non-contractible component of $C$ is separating in $R$, by Proposition \ref{prop_embedding_of_annulus_in_R} and the computation of $\AHI$ in \cite[Example 4.2]{AHI}, we have $\dim \SiHI(\ga_i)=2$. Hence the lemma is proved.
\end{proof}

For the rest of this subsection, we define an explicit isomorphism from $V_\bC(C)$ to $\SiHI(C)$ that realizes the isomorphism in Lemma \ref{lem_SiHI(C)_iso_V(C)}.
The isomorphism is constructed by the following steps.

\subsubsection{Step 1}
\label{subsubsec_iso_V_to_SiHI_1}
 We first fix a basis for $\SiHI(\ga)$ when $\ga$ is contractible. Let $ u_0\in \MHI_{R,p}(\emptyset)$ be the element given by Lemma \ref{lem_multiplicative_unit}.  The element $u_0$ is unique up to sign.

If $\ga$ is a contractible simple closed curve on $\{0\}\times R$, then it bounds a unique disk $B$ in $\{0\}\times R$. The disk $B$ defines a link cobordism from the empty set to $\ga$.  Denote this link cobordism by $D^+(\ga)$.  Let $E^+(\ga)$ be the connected sum of $D^+(\ga)$ with a standard torus in a small ball. 

Define $v_+(\ga)\in \SiHI(\ga)$ to be the image of $u_0$ under the cobordism map of $\SiHI_{R,p}(D^+)$, define $v_-(\ga)$ to be the image of $u_0$ under the map $\frac 12 \SiHI_{R,p}(E^+)$. By definition, $v_+(\ga)$ and $v_-(\ga)$ are only well defined up to signs. 

\begin{lem}
	\label{lem_linear_independence_v_pm}
	$v_+(\ga)$ and $v_-(\ga)$ are linearly independent in $\SiHI(\ga)$.
\end{lem}
\begin{proof}
	Since $\SiHI(\emptyset)\cong \bC$, $\SiHI(\ga)\cong\bC^2$, the lemma is equivalent to the statement that $\SiHI_{R,p}(D^+)$ and $\SiHI_{R,p}(E^+)$ are linearly independent maps from  $\SiHI(\emptyset)$ to $\SiHI(\ga)$. 
	Since $\ga$ is a contractible circle, we can find an embedded non-separating annulus in $R$ that contains $\ga$, and view $D^+(\ga)$ and $E^+(\ga)$ as cobordisms between links in the thickened annulus.
	By Corollary \ref{cor_zero_maps_detected_by_AHI}, we only need to show that $\AHI(D^+(\ga))$ and $\AHI(E^+(\ga))$ are linearly independent as maps from $\AHI(\emptyset)$ to $\AHI(\ga)$.  This is proved in \cite{AHI}. In the notation of \cite{AHI}, these two cobordism maps are denoted by $\AHI(D^+)$ and $\AHI(\Sigma^+)$, and the desired result follows from \cite[Proposition 5.9]{AHI}.
\end{proof}

The mirror image of $D^+(\ga)$ defines a cobordism from $\ga$ to the empty set.  We denote this cobordism by $D^-(\ga)$. 
\begin{lem}
	\label{lem_SiHI_D-}
	The map $\SiHI_{R,p}(D^-(\ga))$ takes $v_+(\ga)$ to $0$, takes $v_-(\ga)$ to $\pm u_0$.
\end{lem}
\begin{proof}
	This follows from the same argument as Lemma \ref{lem_linear_independence_v_pm}. We use Corollary \ref{cor_zero_maps_detected_by_AHI} to reduce to the case of $\AHI$, and then cite \cite[Proposition 5.9]{AHI}. The  cobordism map defined by $D^-(\ga)$ is denoted by $\AHI(D^-)$ in \cite{AHI}.
\end{proof}

\subsubsection{Step 2.} 
\label{subsubsec_iso_V_to_SiHI_2}
Now we fix a basis for $\SiHI(\ga)$ when $\ga$ is non-separating.

Fix a pair of non-separating oriented simple closed curves $(\ga_0,\ga_0')$ on $R$ with algebraic intersection number $1$ that are disjoint from $p$. By Lemma \ref{lem_grading_SiHI_range},  the $\ga_0'$--degrees of $\SiHI(\ga_0)$ are contained in $\{-1,1\}$. By Lemma \ref{lem_SiHI(C)_iso_V(C)}, we have $\dim \SiHI(\ga_0)=2$. 
Hence by Lemma \ref{lem_same_dim_opposite_eigenvalues}, for each $d\in\{1,-1\}$, the component of $\SiHI_{\Sigma,p}(\ga_0)$ with $\ga_0'$--degree $d$ is one dimensional.

 Choose a non-zero homogeneous element $w_0\in \SiHI(\ga_0)$  with respect to the $\ga_0'$--degree. By Proposition \ref{prop_diff_induces_id_on_II}, this condition on $w_0$ does not depend on the choice of $\ga_0'$.
 
 \begin{lem}
 	\label{lem_ga0_to_ga_SiHI}
 	Suppose $\ga$ is an oriented non-separating simple closed curve on $R$ and let $\varphi:R\to R$ be a diffeomorphism fixing $p$ such that $\varphi(\ga_0) = \ga$. Let $\varphi_*:\SiHI(\ga_0)\to \SiHI(\ga)$ be the isomorphism induced by the diffeomorphism $\id_{S^1}\times \varphi$. Then the map $\varphi_*$ is independent of the choice of $\varphi$ up to sign. 
 \end{lem}

\begin{proof}
	Suppose $\varphi_1$ and $\varphi_2$ are two diffeomorphisms on $R$ such that $\varphi_i(p)=p$ and $\varphi_i(\ga_0)=\ga$. We need to show that $(\varphi_1)_* = \pm (\varphi_2)_*$. 
	
	The map $\varphi_1^{-1}\circ \varphi_2$ can be written as a composition of two diffeomorphisms $\phi_1$ and $\phi_2$, where $\phi_1$ is isotopic to the identity, and $\phi_2$ is the identity map on a neighborhood of $\ga_0$. By Proposition \ref{prop_diff_induces_id_on_II}, the map defined by $\phi_2$ is $\pm \id$. Since $\phi_1$ is isotopic to the identity, it also induced $\pm \id$ on $\SiHI(\ga_0)$ (see Example \ref{exmp_diff_induce_iso}). Therefore $(\varphi_1)_* = \pm (\varphi_2)_*$. 
\end{proof}

\begin{defn}
	Suppose $\ga$ is a non-separating curve on $R$ disjoint from $p$, and let $\mfo$ be an orientation of $\ga$. 
    Let $\varphi$ be an orientation-preserving self-diffeomorphism of $R$ that takes $\ga_0$ to $\ga$ with orientation $\mfo$, let $\varphi_*$ be defined as in Lemma \ref{lem_ga0_to_ga_SiHI}.
	Define $w_{\mfo,w_0}(\ga)\in \SiHI(\ga)$ to be the image of $w_0$ under the map $\varphi_*$.
\end{defn}

By definition, $w_{\mfo,w_0}(\ga)$ is well-defined up to sign.

\begin{lem}
	Suppose $\ga$ is a non-separating curve on $R$ disjoint from $p$, and let $\mfo$, $\mfo'$ be the two orientations of $\ga$. Then $w_{\mfo,w_0}(\ga)$ and $w_{\mfo',w_0}(\ga)$ are linearly independent in $\SiHI(\ga)$.
\end{lem}

\begin{proof}
	Let $\ga'$ be a fixed oriented non-separating curve on $R$ that intersects $\ga$ transversely at one point. Then by definition, $w_{\mfo,w_0}(\ga)$ and $w_{\mfo',w_0}(\ga)$ are both homogeneous with respect to the $\ga'$--degree, and their $\ga'$--degrees have opposite signs. Therefore $w_{\mfo,w_0}(\ga)$ and $w_{\mfo',w_0}(\ga)$ are linearly independent.
\end{proof}

\subsubsection{Step 3}
Now we combine the results above to define an explicit isomorphism from $V_\bC(C)$ to $\SiHI(C)$.

\begin{defn}
	\begin{enumerate}
		\item 	If $\ga$ is a contractible simple closed curve on $R$, define
		$\theta_{w_0}(\ga): V_\bC(\ga)\to \SiHI(\ga)$
		to be the linear map that takes $\bfv_+(\ga)$ to $v_+(\ga)$, and takes $\bfv_-(\ga)$ to $v_-(\ga)$.
		\item 	If $\ga$ is a non-separating simple closed curve on $R$, define
		$\theta_{w_0}(\ga): V_\bC(\ga)\to \SiHI(\ga)$
		to be the linear map that takes $\bfw_\mfo(\ga)$ to $w_{\mfo,w_0}(\ga)$ for each orientation $\mfo$ of $\ga$.
	\end{enumerate}
\end{defn}

The maps $\theta_{w_0}(\ga)$ are isomorphisms. They are only well defined up to changes of signs on the images of the standard basis of $V_\bC(\ga)$. When $\ga$ is contractible,  $\theta_{w_0}(\ga)$ does not depend on the choice of $w_0$.

Suppose $C\subset R\backslash\{p\}$ is a closed $1$--manifold such that every non-trivial component is non-separating. Let $\sigma$ be an ordering of the components of $C$ as $(\ga_1,\ga_2,\cdots,\ga_k)$. 
Then the link $C$ is isotopic to 
$$(\{0\}\times\ga_1)\sqcup \dots\sqcup (\{0\}\times \ga_k)$$
 via an isotopy of $(-1,1)\times R$ that fixes the projection of every point on $R$. Any two such isotopies induce the same map from $\SiHI(C)$ to 
 $$\SiHI_{R,p}((\{0\}\times\ga_1)\sqcup \dots\sqcup (\{0\}\times \ga_k))$$ 
 because the diffeomorphisms defining the two maps are isotopic.
\begin{defn}
	\label{defn_Theta}
Define
$$
\Theta_{w_0,\sigma} : V_\bC(C) \to \SiHI(C)
$$
by $\Theta_{w_0,\sigma}  = \clP\circ (\theta_{w_0}(\ga_1)\otimes \dots \otimes \theta_{w_0}(\ga_k))$, where we identify $\SiHI(C)$ with $\SiHI_{R,p}((\{0\}\times\ga_1)\sqcup \dots\sqcup (\{0\}\times \ga_k))$ using the isomorphism described above, and $\clP$ is the map defined in Definition \ref{defn_clP}.
\end{defn}

\subsection{Properties of $\Theta_{w_0,\sigma}$}

The next two lemmas study the behavior of the map $\Theta_{w_0,\sigma}$ when changing the ordering of the components of $C$. 

\begin{lem}
	\label{lem_change_order_with_contractible}
	Suppose $\ga_1,\ga_2$ are two disjoint curves on $R\backslash\{p\}$ such that $\ga_1$ is contractible, and $\ga_2$ is either contractible or non-separating. Let $x\in V(\ga_2)$ be an element in the standard basis.  In other words, $x$ is of the form $\bfw_{\mfo}(\ga_2)$ if $\ga_2$ is non-separating, and is of the form $\bfv_\pm(\ga_2)$ if $\ga_2$ is contractible. Let $C=\ga_1\cup \ga_2$, let $\sigma$ be the ordering $(\ga_1,\ga_2)$, and $\sigma'$ be the ordering $(\ga_2,\ga_1)$. Then for each $\epsilon\in\{+,-\}$,
	$$
	 \Theta_{w_0,\sigma} \big( \bfv_\epsilon(\ga_1) \otimes x\big) = \pm   \Theta_{w_0,\sigma'} \big( x\otimes \bfv_\epsilon(\ga_1) \big).
	$$
\end{lem}

\begin{proof}
Each side of the equation is given by the image of a cobordism map from 
$$(S^1\times R, \emptyset, S^1\times \{p\} ) \sqcup (S^1\times R, \{0\}\times \ga_2, S^1\times \{p\} )$$
to 
$$(S^1\times R, \{0\}\times C, S^1\times \{p\} )$$
applied to the element $u_0\otimes x\in \SiHI(\emptyset)\otimes \SiHI(\ga_2)$. It is straightforward to verify that the two cobordisms are diffeomorphic relative to the boundary.
\end{proof}

\begin{lem}
	\label{lem_change_order_non_parallel}
	Suppose $\ga_1,\ga_2$ are two disjoint curves on $R$ such that $R\backslash{\ga_1\cup \ga_2}$ is connected. Let $\mfo_1,\mfo_2$ be fixed orientations of $\ga_1,\ga_2$. Let $C=\ga_1\cup \ga_2$, let $\sigma$ be the ordering $(\ga_1,\ga_2)$, and $\sigma'$ be the ordering $(\ga_2,\ga_1)$. Then there exists $\lambda=\pm 1$ or $\pm i$, such that 
	$$
	 \Theta_{w_0,\sigma} \big( \bfw_{\mfo_1}(\ga_1) \otimes \bfw_{\mfo_2}(\ga_2)\big) =  \lambda \cdot \Theta_{w_0,\sigma'} \big(  \bfw_{\mfo_2}(\ga_2) \otimes \bfw_{\mfo_1}(\ga_1)\big).
	$$
\end{lem}

\begin{proof}
	Let $\ga_1'$ be an oriented simple closed curve that is disjoint from $\ga_2$ and intersects $\ga_1$ transversely at one point, and let $\ga_2'$ be an oriented simple closed curve that is disjoint from $\ga_1$ and intersects $\ga_2$ transversely at one point. Let $\mfo_1'$, $\mfo_2'$ be the opposite orientations of $\mfo_1$ and $\mfo_2$ respectively. Then $\SiHI(C)$ is spanned by the images under $\Theta_{w_0,\sigma} $ of 
	$$ \bfw_{\mfo_1}(\ga_1) \otimes \bfw_{\mfo_2}(\ga_2),\quad
	 \bfw_{\mfo_1'}(\ga_1) \otimes \bfw_{\mfo_2}(\ga_2), \quad
	 \bfw_{\mfo_1}(\ga_1) \otimes \bfw_{\mfo_2'}(\ga_2),\quad
	 \bfw_{\mfo_1'}(\ga_1) \otimes \bfw_{\mfo_2'}(\ga_2),$$
	and the four images have different gradings with respect to the pair $(\ga_1',\ga_2')$. The same statement holds for $\Theta_{w_0,\sigma'}$. 
Comparing the gradings of the images shows that 
		$$		 \Theta_{w_0,\sigma} \big( \bfw_{\mfo_1}(\ga_1) \otimes \bfw_{\mfo_2}(\ga_2)\big) =  \lambda \cdot \Theta_{w_0,\sigma'} \big(  \bfw_{\mfo_2}(\ga_2) \otimes \bfw_{\mfo_1}(\ga_1)\big)
		$$
	for some non-zero constant $\lambda\in \bC$ that is well-defined up to sign.
		There is an orientation-preserving diffeomorphism of $R$ that switches $(\ga_1,\mfo_1)$ and $(\ga_2,\mfo_2)$, so we must have $\lambda=\pm \lambda^{-1}$.
 Therefore $\lambda$ is an integer power of $i$.
\end{proof}

We also need the following technical results.

\begin{lem}
	\label{lem_cobordism_map_symmetry_1}
	Suppose $\ga_1,\ga_2$ are parallel simple closed curves on $R$ that are non-separating. Let $C= \ga_1\cup \ga_2$, and $S\subset \Sigma$ be the annulus contained in $\Sigma$ such that $\partial S = C$. Then  $S$ defines a link cobordism from $\emptyset$ to $C$. 
	Since $\ga_1$ and $\ga_2$ are parallel, we identify the orientations of $\ga_1$ and $\ga_2$. Let $\mfo,\mfo'$ be the two orientations of $\ga_i$ for $i=1,2$. 
	Let $\sigma = (\ga_1,\ga_2)$.
	Then there exists $\lambda\in \bC$ such that 
	$$
\SiHI_{R,p}(S)(u_0) = \pm	\Theta_{w_0,\sigma} \big( \lambda\cdot  \bfw_{\mfo}(\ga_1) \otimes \bfw_{\mfo'}(\ga_2) \pm  \lambda\cdot  \bfw_{\mfo'}(\ga_1) \otimes \cdot  \bfw_{\mfo}(\ga_2) \big).
	$$
\end{lem}

\begin{proof}
	For each simple closed curve $\ga'$ on $R$, the $\ga'$--grading of the left-hand-side of the equation is zero. Therefore, $\SiHI_{R,p}(S)(u_0)$ must have the form  
	$$\pm \Theta_{w_0,\sigma} \big( \lambda_1 \cdot  \bfw_{\mfo}(\ga_1) \otimes \bfw_{\mfo'}(\ga_2)+  \lambda_2 \cdot  \bfw_{\mfo'}(\ga_1) \otimes \cdot  \bfw_{\mfo}(\ga_2) \big).$$
	We need to show that $\lambda_1=\pm \lambda_2$. 
	
	Since $\ga_1$ and $\ga_2$ are parallel, there is an isomorphism $\iota: \SiHI(\ga_1)\to \SiHI(\ga_2)$ defined by the isotopy from $\ga_1$ to $\ga_2$. 
	Consider the element 
	$$
	\xi = \mathcal{P}^{-1}\circ  \SiHI_{R,p}(S)(u_0) \in \SiHI(\ga_1)\otimes \SiHI(\ga_2),
	$$
	where we identify $\SiHI(\ga_1\cup\ga_2)$ with $\SiHI_{R,p}(\{0\}\times \ga_1\sqcup \{0\}\times \ga_2)$ as in Definition \ref{defn_Theta}. 
	Let $\phi:\SiHI(\ga_1)\otimes \SiHI(\ga_2)\to \SiHI(\ga_2)\otimes \SiHI(\ga_1)$ be the isomorphism that switches the two components. It suffices to show that
	\begin{equation}
		\label{eqn_phi(xi)=xi}
	\phi(\xi) = \pm (\iota\otimes \iota^{-1})\xi.
	\end{equation}
	
	Let $\hat \clP: \SiHI_{R,p}(\ga_1\cup \ga_2)\to \SiHI(\ga_1)\otimes \SiHI(\ga_2)$ be the  map induced by the reversed pair-of-pants cobordism. Recall that the excision theorem was proved by showing that there exists a non-zero constant $c$ only depending on the genus of $R$ such that $\clP\circ  \hat \clP = c\cdot \id.$  So we have $\clP ^{-1} = c^{-1}\cdot \hat\clP$ for some constant $c$ that only depends on the genus of $R$. 
	
	We claim that
	\begin{equation}
		\label{eqn_empty_to_two_parallel_inv_by_switch}
			(\iota\otimes \iota^{-1})\circ \hat{\mathcal{P}}\circ  \SiHI_{R,p}(S) = \pm \phi \circ \hat{ \mathcal{P}}\circ  \SiHI_{R,p}(S)
	\end{equation}
	as maps from $\SiHI(\emptyset)$ to $\SiHI(\ga_2)\otimes \SiHI(\ga_1)$. Equation \eqref{eqn_empty_to_two_parallel_inv_by_switch} holds because the cobordisms defining both sides of the equation are diffeomorphic.  In fact, the cobordisms defining the two sides of \eqref{eqn_empty_to_two_parallel_inv_by_switch} are shown schematically in Figure \ref{fig_pair_of_pants} and Figure \ref{fig_pair_of_pants_switch} in the direction perpendicular to $R$, and the identity follows from the fact that there is a diffeomorphism from Figure \ref{fig_pair_of_pants} to Figure \ref{fig_pair_of_pants_switch} that restricts to the identity on the boundary. Equation \eqref{eqn_empty_to_two_parallel_inv_by_switch} immediately implies \eqref{eqn_phi(xi)=xi}, and the desired result is proved.
\end{proof}
	
	\begin{figure}[!htb]
		\begin{minipage}{0.48\textwidth}
			\centering
			\includegraphics[width=.3\linewidth]{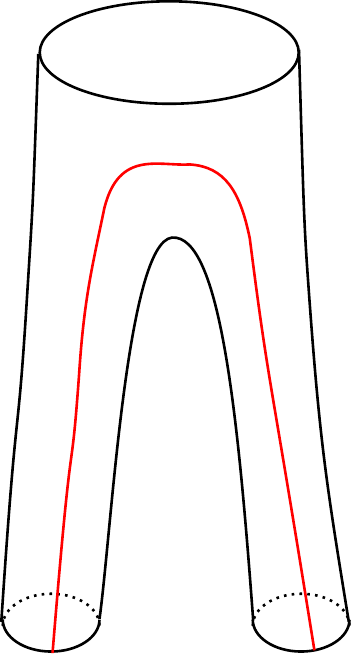}
			\caption{}
			\label{fig_pair_of_pants}
		\end{minipage}\hfill
		\begin{minipage}{0.48\textwidth}
			\centering
			\includegraphics[width=.3\linewidth]{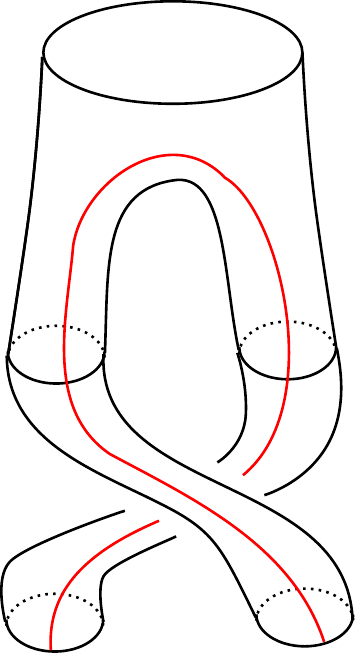}
			\caption{}
		   \label{fig_pair_of_pants_switch}
		\end{minipage}
	\end{figure}

\begin{lem}
	\label{lem_cobordism_map_symmetry_2}
	Suppose $\ga_1,\ga_2$ are parallel simple closed curves on $R$ that are non-separating. Let $C= \ga_1\cup \ga_2$, and $S\subset \Sigma$ be the annulus contained in $\Sigma$ such that $\partial S = C$. Then  $S$ defines a link cobordism from $C$ to $\emptyset$. 
Since $\ga_1$ and $\ga_2$ are parallel, we identify the orientations of $\ga_1$ and $\ga_2$. Let $\mfo,\mfo'$ be the two orientations of $\ga_i$ for $i=1,2$. 
Let $\sigma = (\ga_1,\ga_2)$.
	Then 
	$$
\SiHI_{R,p}(S)  \circ \Theta_{w_0,\sigma}( \bfw_{\mfo}(\ga_1) \otimes \bfw_{\mfo'}(\ga_2) )= \pm \SiHI_{R,p}(S) \circ  \Theta_{w_0,\sigma}( \bfw_{\mfo'}(\ga_1) \otimes \bfw_{\mfo}(\ga_2) )
	$$
\end{lem}

\begin{proof}
This follows from the same argument as Lemma \ref{lem_cobordism_map_symmetry_1} by reversing the directions of all cobordisms.
\end{proof}

\section{Proof of the main theorem}
\label{sec_proof_main}

Let $\Sigma$ be a fixed compact surface with genus zero.
Recall that when $\Sigma\cong S^2$, Theorem \ref{thm_main} follows from Kronheimer--Mrowka's unknot detection result. Therefore we will always assume that $\partial \Sigma\neq \emptyset$ in this section.

Let $F$ be a connected surface such that there is an orientation-reversing diffeomorphism $\tau: \partial F \to \partial \Sigma$. Fix the map $\tau$ and let $R = F\cup_\tau N$. Choose $F$ so that the genus of $R$ is at least $3$. Let $p$ be a fixed point on $F$.

Let $C$ be an embedded closed $1$--manifold in $\Sigma$, and view it as a submanifold of $R$. Since $\Sigma$ has genus zero and $F$ is connected, every non-contractible simple closed curve on $C$ is non-separating in $R$. So the discussions from Section \ref{sec_instanton_for_links_at_0} can be applied to $C$. Let $w_0$ be given as in Section \ref{subsec_iso_V_to_SiHI}.

Suppose $b$ is an embedded disk on $\Sigma$ such that the interior of $b$ is disjoint from $C$ and the boundary of $b$ intersects $C$ at two arcs. 
Then surgery of $C$ along $b$ generates another embedded closed $1$--manifold on $\Sigma$, which we denote by $C_b$. See Figure \ref{fig:band_surgery} for an illustration of the surgery. We will call the disk $b$ the ``band'' that is \emph{attached} to $C$ and call the surgery from $C$ to $C_b$ the \emph{band surgery along $b$}.

\begin{figure}
	\begin{overpic}[width=0.7\textwidth]{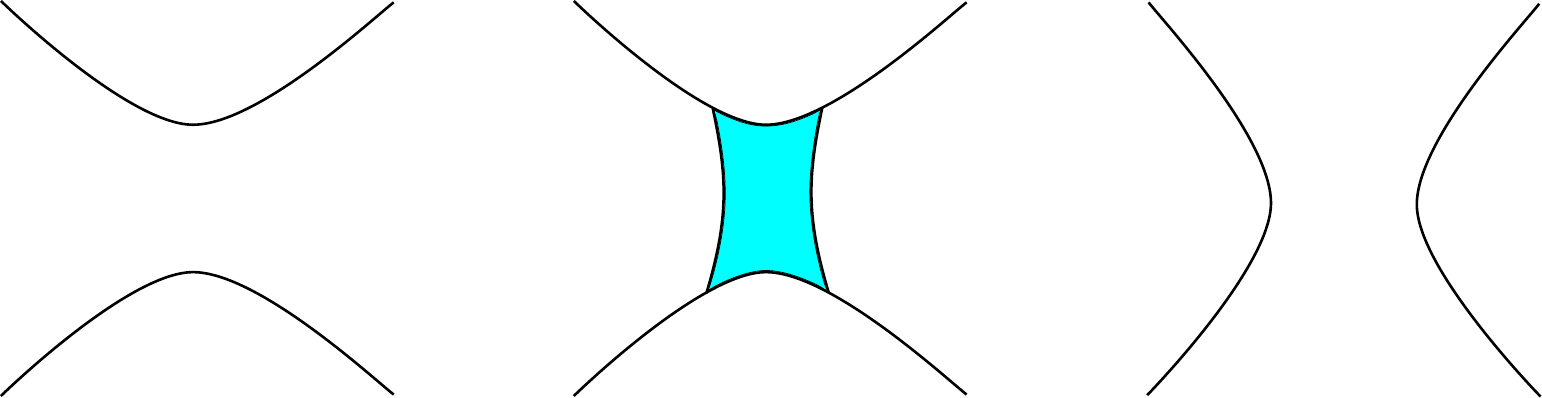}
		\put(2,-5){the manifold $C$}
		\put(43,-5){the band $b$}
		\put(77,-5){the manifold $C_b$}
	\end{overpic}
	\vspace{\baselineskip}
	\caption{Band surgery}\label{fig:band_surgery}
\end{figure}

The band surgery along $b$ defines a link cobordism from $C$ to $C_b$.
Let $\SiHI(b)$ denote the cobordism map from $\SiHI(C)$ to $\SiHI(C_b)$.
We will study the properties of the map $\SiHI(b)$. 

To start, consider the case where one of $C$ or $C_b$ has exactly one connected component. 

In the case that $C$ has one component $\ga$ and $C_b$ has two components $\ga_1$ and $\ga_2$, let $\sigma$ denote the only possible ordering of components of $C$, and let $\sigma_b$ be an arbitrary ordering of the components of $C_b$. 
We have the following results:
\begin{lem}
	\label{lem_cobordism_SiHI_one_trivial_to_two_trivial}
	If $\gamma_1$ and $\gamma_2$ are both trivial circles, then $\ga$ is also trivial. The map $\SiHI(b)$ has the form
	\begin{align*}
		\Theta_{w_0,\sigma}\,\mathbf{v}(\gamma)_+  &\mapsto  \pm \Theta_{w_0,\sigma_b}\big(\mathbf{v}(\gamma_1)_+\otimes \mathbf{v}(\gamma_2)_-  \big) \pm
		\Theta_{w_0,\sigma_b}\big( \mathbf{v}(\gamma_1)_-\otimes \mathbf{v}(\gamma_2)_+\big), \\
	\Theta_{w_0,\sigma}\,	\mathbf{v}(\gamma)_-  &\mapsto  \pm \Theta_{w_0,\sigma_b}\big(\mathbf{v}(\gamma_1)_-\otimes \mathbf{v}(\gamma_2)_-\big).
	\end{align*}
\end{lem} 

\begin{proof}
Recall from Section \ref{subsec_iso_V_to_SiHI} 
that $ \Theta_{w_0,\sigma_b}\big(\mathbf{v}(\gamma_1)_+\otimes \mathbf{v}(\gamma_2)_-  \big)$ is given by the image of 
$$\SiHI(E^+(\ga_1))(u_0)\otimes \SiHI(D^+(\ga_2))(u_0)$$
under the pair-of-pants map $\clP$. 
Therefore, $ \Theta_{w_0,\sigma_b}\big(\mathbf{v}(\gamma_1)_+\otimes \mathbf{v}(\gamma_2)_-  \big)$ is defined by a cobordism $(\hat W,\hat S,\hat \Omega)$ from $(S^1\times R, \emptyset, S^1\times \{p\})\sqcup (S^1\times R, \emptyset, S^1\times \{p\})$ to $(S^1\times R, \{0\}\times (\ga_1\sqcup \ga_2), S^1\times \{p\})$.
Let $S$ be the cobordism from $\emptyset$ to $\ga_1\cup \ga_2$ given by the disjoint union of $E^+(\ga_1)$ and $D^+(\ga_2)$. 
Then the cobordism $\hat S$ can be pushed to a tubular neighborhood of the boundary $(S^1\times R, \{0\}\times (\ga_1\cup \ga_2), S^1\times \{p\})$ to the position of $S$.
Since $\clP(u_0\otimes u_0)=\pm u_0$, this proves
$$ \Theta_{w_0,\sigma_b}\big(\mathbf{v}(\gamma_1)_+\otimes \mathbf{v}(\gamma_2)_-  \big) = \pm \SiHI(S)(u_0).$$
Similarly, the following elements in $\SiHI(\ga_1\cup\ga_2)$:
$$\Theta_{w_0,\sigma_b}\big( \mathbf{v}(\gamma_1)_-\otimes \mathbf{v}(\gamma_2)_+\big),
\;
\Theta_{w_0,\sigma_b}\big(\mathbf{v}(\gamma_1)_-\otimes \mathbf{v}(\gamma_2)_-\big),
\;
\SiHI(b)\circ \Theta_{w_0,\sigma}\,\mathbf{v}(\gamma)_+ ,
\;
\SiHI(b)\circ \Theta_{w_0,\sigma}\,\mathbf{v}(\gamma)_-,$$ are all given by the image of $u_0$ under maps induced from suitable cobordisms between $\emptyset$ and $\{0\}\times (\ga_1\cup \ga_2)$. 

Since $\SiHI(\emptyset)\cong \bC$, the desired lemma is equivalent to the corresponding linear relations on these cobordism maps.
By Corollary \ref{cor_zero_maps_detected_by_AHI}, we only need to prove the analogous result in $\AHI$. Hence the result follows from \cite[Proposition 5.9]{AHI}.
\end{proof}

\begin{lem}
	\label{lem_cobordism_SiHI_non_contractible_split_one_contractible}
	If one of $\{\ga_1,\ga_2\}$ is contractible and the other is non-contractible, assume without loss of generality that $\ga_1$ is contractible and $\ga_2$ is non-contractible. Then $\ga$ is isotopic to $\ga_2$, and the orientations of $\ga$ and $\ga_2$ are canonically identified.  Let $\mfo$ be an arbitrary orientation of $\ga$, and use the same notation to denote the corresponding orientation of $\ga_2$. Then $\SiHI(b)$ has the form
	$$
	\Theta_{w_0,\sigma}\,\mathbf{w}_{\mfo} (\gamma)\mapsto \pm \Theta_{w_0,\sigma_b}(\mathbf{v}_-(\gamma_1)\otimes \mathbf{w}_{\mfo}(\gamma_2)).
	$$
\end{lem}

\begin{proof}
Similar to the argument of Lemma \ref{lem_cobordism_SiHI_one_trivial_to_two_trivial}, both sides of the equations can be interpreted as the images of suitable cobordism maps from $\SiHI(\ga)$ to $\SiHI(\ga_1\cup \ga_2)$ applied to the element  $w_{\mfo,w_0}(\ga)\in \SiHI(\ga)$ (see Section \ref{subsubsec_iso_V_to_SiHI_2} for the definition of $w_{\mfo,w_0}(\ga)$). Therefore we only need to prove the equation for the cobordism maps. The result then follows from Corollary \ref{cor_zero_maps_detected_by_AHI} and \cite[Proposition 5.14]{AHI}.
\end{proof}

\begin{lem}
	\label{lem_SiHI_cobordism_one_trivial_two_nontrivial}
	If both $\ga_1$ and $\ga_2$ are non-contractible and $\ga$ is contractible. Then $\ga_1$, $\ga_2$ are isotopic to each other, and the orientations of $\ga_1$ are $\ga_2$ are canonically identified. Let $\mfo,\mfo'$ be the orientations of $\ga_1,$ and use the same notation for the orientations of $\ga_2$. Then there exists $\lambda_1\neq 0$, such that the map $\SiHI(b)$ has the form
	$$
	\Theta_{w_0,\sigma}\,\mathbf{v}_+(\gamma) \mapsto \pm  \lambda_1\Theta_{w_0,\sigma_b}\big(\mathbf{w}_\mfo(\gamma_1)\otimes\mathbf{w}_{\mfo'}(\gamma_2)\big)\pm \Theta_{w_0,\sigma_b}\big( \lambda_1 \mathbf{w}_{\mfo'}(\gamma_1)\otimes \mathbf{w}_{\mfo} (\gamma_2)\big), \quad \Theta_{w_0,\sigma}\,\mathbf{v}_-(\gamma)\mapsto 0.
	$$
\end{lem}
\begin{proof}
	The image of $\Theta_{w_0,\sigma}\,\bfv_-(\ga)$ is zero and the image of $\Theta_{w_0,\sigma}\, \bfv_+(\ga)$ is non-zero because of Corollary \ref{cor_zero_maps_detected_by_AHI} and  \cite[Proposition 5.14]{AHI}. The image of  $\Theta_{w_0,\sigma}\, \bfv_+(\ga)$ is given by the desired the form because of Lemma \ref{lem_cobordism_map_symmetry_1}. 
\end{proof}

\begin{lem}
	\label{lem_cobordism_map_1_nontrivial_to_2_nontrivial}
	If all of $\ga,\ga_1,\ga_2$ are non-contractible, let $N$ be the regular neighborhood of $b\cup \ga$.  Then $N$ is a sphere with three disks removed, and the three boundaries of $N$ are parallel to $\ga_1$, $\ga_2$, $\ga$.  The boundary orientation of $N$ defines an orientation on each of $\ga_1, \ga_2, \ga$, and we denote them by $\mfo_1,\mfo_2,\mfo$ respectively.  Denote their opposite orientations by $\mfo_1',\mfo_2',\mfo'$. Then there exists a choice of $w_0$ such that $\SiHI(b)$ has the form
	\[
	\Theta_{w_0,\sigma}\,	\mathbf{w}_{\mfo'}(\gamma)\mapsto \pm  \lambda_2\, \Theta_{w_0,\sigma_b}\big(\bfw(\ga_1)_{\mfo_1}\otimes \mathbf{w}(\gamma_2)_{\mfo_2}\big),\quad 	\Theta_{w_0,\sigma}\, \bfw(\ga)_{\mfo}\mapsto 0.
	\]
	for some constant $\lambda_2$.
\end{lem}

\begin{proof}
		By comparing gradings, the image of $\Theta_{w_0,\sigma}\,	\mathbf{w}_{\mfo'}(\gamma)$ must be a scalar multiple of $$\Theta_{w_0,\sigma_b}\big(\bfw(\ga_1)_{\mfo_1}\otimes \bfw(\gamma_2)_{\mfo_2}\big),$$ and  the image of $\Theta_{w_0,\sigma}\,	\mathbf{w}_{\mfo}(\gamma)$ must be a scalar multiple of $$\Theta_{w_0,\sigma_b}\big(\bfw(\ga_1)_{\mfo_1'}\otimes \bfw(\gamma_2)_{\mfo_2'}\big).$$
	
	We show that the rank of $\SiHI(b)$ must be strictly less than $2$.
	Since cobordism maps are equivariant with respect to orientation preserving self-diffeomorphisms of $R$, we may assume without loss of generality that $\ga$ and $b$ are given by Figure \ref{fig_cobordism_3_non_contractible}, where the boundaries of black dots are non-contractible curves in $\Sigma$. 

	\begin{figure}[!htb]
	\begin{minipage}{0.48\textwidth}
		\centering
		\begin{overpic}[width=.8\linewidth]{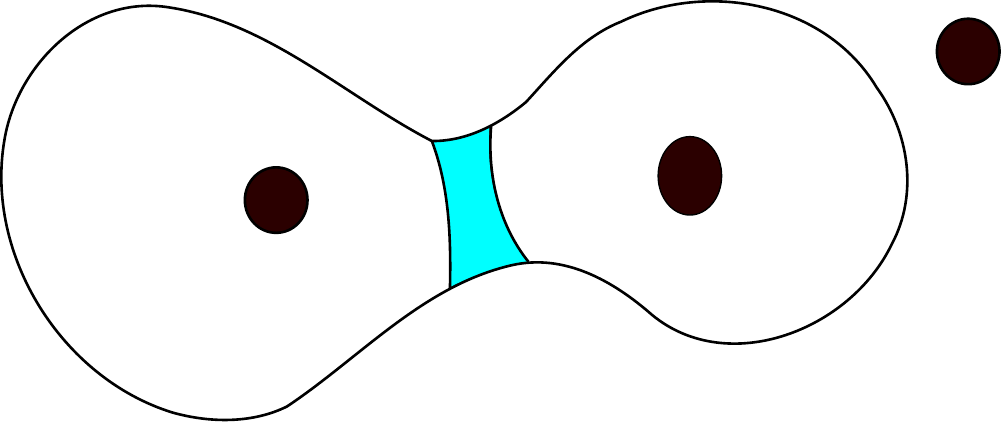}
			\put(2,4){$\ga$}
 			\put(48,9){$b$}
		\end{overpic}
		\caption{}
		\label{fig_cobordism_3_non_contractible}
	\end{minipage}\hfill
	\begin{minipage}{0.48\textwidth}
		\centering
		\includegraphics[width=.7\linewidth]{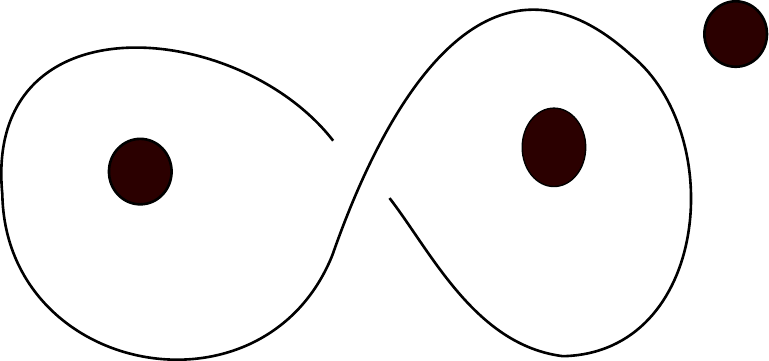}
		\caption{A knot in $(-1,1)\times \Sigma$}
		\label{fig_a_knot_on_sigma}
	\end{minipage}
\end{figure}

Consider the knot $K$ in the interior of $[-1,1]\times \Sigma$ as shown in Figure \ref{fig_a_knot_on_sigma}. The fundamental class of $K$ is not in the image of any embedded annulus in $\Sigma$, so by Theorem \ref{thm_main_instanton}, we have 
$$
\dim \SHI([-1,1]\times \Sigma, \{0\}\times \Sigma, K) > 2.
$$
By \eqref{eqn_SiHI_iso_SHI}, this implies $\dim \SiHI_{R,p}(K)>2$. By Kronheimer--Mrowka's unoriented skein exact triangle 
\cite[Theorem 6.8]{KM:Kh-unknot}, there is a long exact sequence
$$
\cdots \to \SiHI(\ga) \overset{\SiHI(b)}{\longrightarrow} \SiHI(\ga_1\cup \ga_2) \to \SiHI_{R,p}(K) \to \SiHI(\ga)  \to \SiHI(\ga) \overset{\SiHI(b)}{\longrightarrow} \SiHI(\ga_1\cup \ga_2)\to  \cdots.
$$
Therefore the rank of $\SiHI(b)$ cannot be $2$.
	
By diffeomorphism invariance, either the image of $\Theta_{w_0,\sigma}\bfw_\mfo(\ga)$ is zero for all choices of $\ga,\ga_1,\ga_2$,
	or the image of $\Theta_{w_0,\sigma}\bfw_{\mfo'}(\ga)$ is zero for all choices of $\ga,\ga_1,\ga_2$. Note that in the definition of $w_0$, we may choose it to have grading $1$ or $-1$ with respect to the $\ga_0'$--grading. Suppose $w_0'$ is a different choice of $w_0$ with a different sign in the $\ga_0'$--grading, then $w_{\mfo,w_0}(\ga)$ is a scalar multiple of $w_{\mfo',w_0'}(\ga)$. Therefore,  we can always choose $w_0$ so that the image of $\Theta_{w_0,\sigma}\bfw_{\mfo}(\ga)$ is zero. 
\end{proof}

In the case that $c$ has two components $\ga_1$ and $\ga_2$ and they are merged into one circle $c_b=\ga$ after the surgery, let $\sigma$ be an arbitrary ordering of the components of $C$, let $\sigma_b$ denote the only possible ordering of components of $C_b$. We have the following results.

\begin{lem} If both $\ga_1$ and $\ga_2$ are contractible circles, then $\ga$ is also contractible. In this case, the map $\SiHI(b)$ has the form
	\begin{align*}
		\Theta_{w_0,\sigma}\big(\mathbf{v}_+(c_1) \otimes \mathbf{v}_+(c_2)\big) &\mapsto 
		\pm \Theta_{w_0,\sigma_b}\,\mathbf{v}_+(c),  &\Theta_{w_0,\sigma}\big(\mathbf{v}_+(c_1)& \otimes \mathbf{v}_-(c_2)\big) \mapsto \pm \Theta_{w_0,\sigma_b}\,\mathbf{v}_-(c) ,\\
		\Theta_{w_0,\sigma}\big(\mathbf{v}_-(c_1) \otimes \mathbf{v}_+(c_2) \big)&\mapsto 
		\pm \Theta_{w_0,\sigma_b}\, \mathbf{v}_-(c),  &\Theta_{w_0,\sigma}\big(\mathbf{v}_-(c_1)& \otimes \mathbf{v}_-(c_2)\big) \mapsto 0.
	\end{align*}
\end{lem}
\begin{proof}
	This follows from the same argument as Lemma \ref{lem_cobordism_SiHI_one_trivial_to_two_trivial}.
\end{proof}
\begin{lem} 	
	\label{lem_SiHI_cobordism_two_nontrivial_merge_to_trivial}
	If one of $\{\ga_1,\ga_2\}$ is contractible and the other is non-contractible, assume without loss of generality that $\ga_1$ is contractible and $\ga_2$ is non-contractible. Then $\ga$ is isotopic to $\ga_2$, and the orientations of $\ga$ and $\ga_2$ are canonically identified.  Let $\mfo$ be an arbitrary orientation of $\ga$, and use the same notation to denote the corresponding orientation of $\ga_2$. Then $\SiHI(b)$ has the form
	\[
	\Theta_{w_0,\sigma}\big(\mathbf{v}_+(\gamma_1)\otimes \bfw_{\mfo}(\ga_2) \big)\mapsto \pm\Theta_{w_0,\sigma_b}\, \bfw_{\mfo}(\ga),\quad  
	\Theta_{w_0,\sigma}\big(\bfv(\ga_1)_-\otimes \bfw_{\mfo}(\ga_2)\big) \mapsto 0.
	\]
\end{lem}
\begin{proof}
	This follows from the same argument as Lemma \ref{lem_cobordism_SiHI_non_contractible_split_one_contractible}.
\end{proof}
\begin{lem} If $\gamma_1$ and $\gamma_2$ are both non-contractible and $\gamma_3$ is contractible, then $\ga_1$ and $\ga_2$ must be isotopic. 
	The orientations of $\ga_1$ and $\ga_2$ are canonically identified by the isotopy. Let $\mfo$, $\mfo'$ be the two orientations of $\ga_1$, and  use the same notation to denote the corresponding orientations of $\ga_2$. 
	Then $\SiHI(b)$ has the form
\begin{align}
	\Theta_{w_0,\sigma}\big(	\mathbf{w}_\mfo(\gamma_1)\otimes \mathbf{w_\mfo}(\gamma_2) \big)&\mapsto 0,   &
	\Theta_{w_0,\sigma}\big(\mathbf{w}_{\mfo'}(\gamma_1)& \otimes \mathbf{w}_{\mfo'}(\gamma_2)\big)\mapsto 0,
	\label{eqn_SiHI_merge_two_nontrivial_to_one_trivial_1}
	\\
	\Theta_{w_0,\sigma}\big(\mathbf{w}_\mfo(\gamma_1)\otimes \mathbf{w}_{\mfo'}(\gamma_2)\big) &\mapsto \pm \lambda_3\Theta_{w_0,\sigma_b}\, \mathbf{v}_-(\gamma), 
	&\Theta_{w_0,\sigma}\big( \mathbf{w}_{\mfo'}(\gamma_1)&\otimes \mathbf{w}_{\mfo} (\gamma_2)\big)\mapsto \pm \lambda_3 \Theta_{w_0,\sigma_b}\, \mathbf{v}_-(\gamma).	
	\label{eqn_SiHI_merge_two_nontrivial_to_one_trivial_2}
\end{align}
for some non-zero $\lambda_3$. 
\end{lem}
\begin{proof}
The right-hand side of the map is a linear combination of $ \Theta_{w_0,\sigma_b}\, \mathbf{v}_-(\gamma)$ and  $\Theta_{w_0,\sigma_b}\, \mathbf{v}_+(\gamma)$. Consider the composition of this map with the splitting of one trivial circle to two trivial circles. By Corollary \ref{cor_zero_maps_detected_by_AHI} and \cite[Proposition 5.9, Proposition 5.14]{AHI}, the composition map is zero. Therefore, by Lemma \ref{lem_cobordism_SiHI_one_trivial_to_two_trivial}, the image of $\SiHI(b)$ is spanned by $ \Theta_{w_0,\sigma_b}\, \mathbf{v}_-(\gamma)$. 

The two terms in \eqref{eqn_SiHI_merge_two_nontrivial_to_one_trivial_1} are zero because of gradings.
To show that the two terms in  \eqref{eqn_SiHI_merge_two_nontrivial_to_one_trivial_2} are equal up to sign, we compose the map with $\SiHI(D^-(\ga))$ from Lemma \ref{lem_SiHI_D-} and invoke Lemma  \ref{lem_cobordism_map_symmetry_2}. The right-hand-side of \eqref{eqn_SiHI_merge_two_nontrivial_to_one_trivial_2} are both non-zero because of Corollary \ref{cor_zero_maps_detected_by_AHI} and \cite[Proposition 5.14]{AHI}.
\end{proof}
\begin{lem} 
	\label{lem_SiHI_cobordism_two_nontrivial_one_nontrivial}
	Assume $w_0$ is chosen so that Lemma \ref{lem_cobordism_map_1_nontrivial_to_2_nontrivial} holds.	If all of $\ga,\ga_1,\ga_2$ are non-contractible, let $N$ be the regular neighborhood of $b\cup \ga$.  Then $N$ is a sphere with three disks removed, and the three boundaries of $N$ are parallel to $\ga_1$, $\ga_2$, $\ga$.  The boundary orientation of $N$ defines an orientation on each of $\ga_1, \ga_2, \ga$, and we denote them by $\mfo_1,\mfo_2,\mfo$ respectively.  Denote their opposite orientations by $\mfo_1',\mfo_2',\mfo'$.  Then $\SiHI(b)$ has the form
	\begin{align*}
			\Theta_{w_0,\sigma}\big( \mathbf{w}_{\mfo_1'}(\gamma_1)\otimes \mathbf{w}_{\mfo_2'}(\gamma_2) )&\mapsto  \pm \lambda_4 \Theta_{w_0,\sigma_b}\,\mathbf{w}_{\mfo}(\gamma),   &\Theta_{w_0,\sigma}\big(\mathbf{w}_{\mfo_1'}(\gamma_1)& \otimes \mathbf{w}_{\mfo_2}(\gamma_2)\big)\mapsto 0,\\
		\Theta_{w_0,\sigma}\big(\mathbf{w}_{\mfo_1}(\gamma_1)\otimes \mathbf{w}_{\mfo_2'}(\gamma_2)\big) &\mapsto 0, &\Theta_{w_0,\sigma}\big(\mathbf{w}_{\mfo_1}(\gamma_1)&\otimes \mathbf{w}_{\mfo_2}(\gamma_2) \big)\mapsto 0.
	\end{align*}
\end{lem}

\begin{proof}
	Comparison of gradings shows that 
	\begin{enumerate}
		\item the image of $\Theta_{w_0,\sigma}\big( \mathbf{v}(\gamma_1)_{\mfo_1'}\otimes \mathbf{v}(\gamma_2)_{\mfo_2'} )$ is a scalar multiple of $ \Theta_{w_0,\sigma_b}\,\mathbf{v}(\gamma)_{\mfo}$,
		\item  the image of $\Theta_{w_0,\sigma}\big( \mathbf{v}(\gamma_1)_{\mfo_1}\otimes \mathbf{v}(\gamma_2)_{\mfo_2} )$
		is a scalar multiple of $ \Theta_{w_0,\sigma_b}\,\mathbf{v}(\gamma)_{\mfo'}$, 
		\item the images of $\Theta_{w_0,\sigma}\big(\mathbf{w}_{\mfo_1'}(\gamma_1) \otimes \mathbf{w}_{\mfo_2}(\gamma_2)\big)$ 
		and
		$\Theta_{w_0,\sigma}\big(\mathbf{w}_{\mfo_1}(\gamma_1)\otimes \mathbf{w}_{\mfo_2'}(\gamma_2)\big) $
		are zero.
	\end{enumerate}
	
	To show that the image of $\Theta_{w_0,\sigma}\big( \mathbf{v}(\gamma_1)_{\mfo_1}\otimes \mathbf{v}(\gamma_2)_{\mfo_2} )$ must also be zero when $w_0$ is chosen so that Lemma \ref{lem_cobordism_map_1_nontrivial_to_2_nontrivial} holds, consider the two bands in Figure \ref{fig_TQFT_two_holes}.  Here, we adopt the convention that if $B$ is a standard disk in the plane, then the induced orientation on $\partial B$ is counter-clockwise. The boundaries of the black dots are non-contractible curves in $\Sigma$. Let $\delta_1,\delta_2$ be simple closed curves on $\Sigma$, let $\mathfrak{r}_1,\mathfrak{r}_2$ be the orientations of $\delta_1,\delta_2$ and $b_1$, $b_2$ be two disjoint bands attached to $\delta_1\cup \delta_2$ as shown in Figure \ref{fig_TQFT_two_holes}. 
	Let $\tau$ denote an arbitrary ordering of the components of $\delta_1\cup \delta_2$. 
	Then by Lemma \ref{lem_cobordism_map_1_nontrivial_to_2_nontrivial}, the image of $\Theta_{w_0,\tau}(\bfw_{\mathfrak{r}_1}(\delta_1)\otimes \bfw_{\mathfrak{r}_2}(\delta_2))$ under $\SiHI(b_2)$ is zero. By the functoriality of $\SiHI$, we have $\SiHI(b_1)\circ \SiHI(b_2) = \SiHI(b_2)\circ \SiHI(b_1)$ on $\SiHI(\delta_1\cup \delta_2)$, therefore
	$$
	 \SiHI(b_2)\circ \SiHI(b_1)\circ \Theta_{w_0,\tau}(\bfw_{\mathfrak{r}_1}(\delta_1)\otimes \bfw_{\mathfrak{r}_2}(\delta_2))=0.
	$$
	It then follows from Lemma \ref{lem_cobordism_SiHI_non_contractible_split_one_contractible} that 
	$$
	\SiHI(b_1)\circ \Theta_{w_0,\tau}(\bfw_{\mathfrak{r}_1}(\delta_1)\otimes \bfw_{\mathfrak{r}_2}(\delta_2))=0.
	$$
	
Since there is an orientation-preserving diffeomorphism of $R$ that takes $(\delta_1,\mathfrak{r}_1,\delta_2,\mathfrak{r}_2,b_1)$ to $(\ga_1,\mfo_1,\ga_2,\mfo_2,b)$, we conclude that the image of $\Theta_{w_0,\sigma}\big(\mathbf{w}_{\mfo_1}(\gamma_1)\otimes \mathbf{w}_{\mfo_2}(\gamma_2)\big)$ under the map $\SiHI(b)$ is zero.
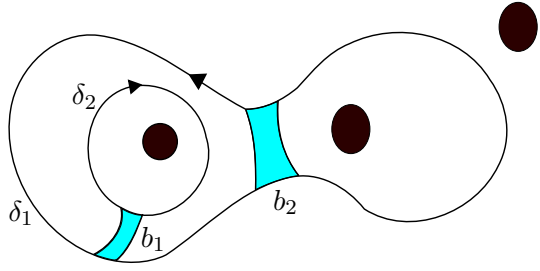
\begin{figure}
		\begin{overpic}[width=0.4\textwidth]{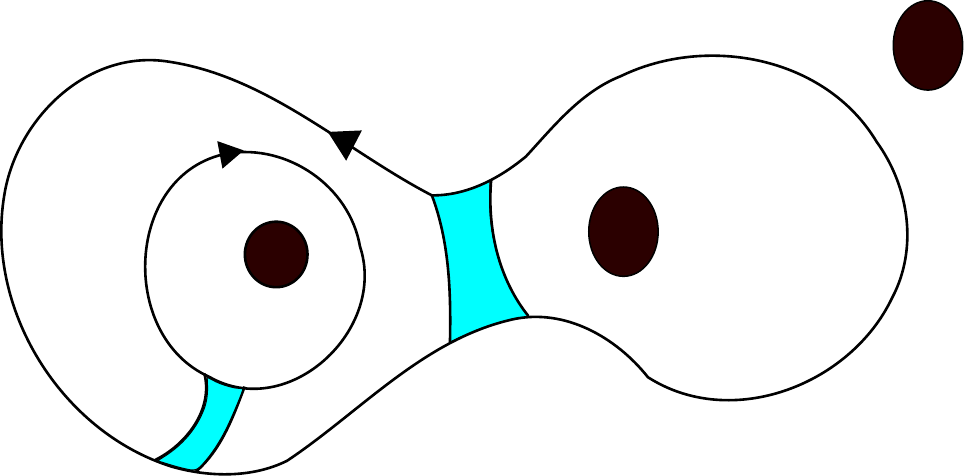}
			\put(0,8){$\delta_1$}
			\put(12,30){$\delta_2$}
			\put(25,3){$b_1$}
			\put(50,10){$b_2$}
		\end{overpic}
		\caption{Two bands}
		\label{fig_TQFT_two_holes}
\end{figure}
\end{proof}

\begin{lem}
	\label{lem_lambda1_lambda3_product}
	Fix a choice of $w_0$. 
	Let $\lambda_1$ be the constant given by Lemma \ref{lem_SiHI_cobordism_one_trivial_two_nontrivial}, and let $\lambda_3$ be the constant given by Lemma \ref{lem_SiHI_cobordism_two_nontrivial_merge_to_trivial}.
	Then $\lambda_1\lambda_3=\pm 1$.
\end{lem}

\begin{proof}
	Consider the composition of two band surgeries, where the first surgery takes one contractible circle $\ga$ to two non-contractible circles $\ga_1$ and $\ga_2$, and the second surgery is defined along the same band and it takes $\ga_1\cup \ga_2$ to $\ga$. By Corollary \ref{cor_zero_maps_detected_by_AHI} and \cite[Proposition 5.14]{AHI}, the composition map  takes $v_-(\ga)$ to $2v_-(\ga)$ (see Section \ref{subsubsec_iso_V_to_SiHI_1} for the definition of $v_-(\ga)$). Hence we must have $\lambda_1\lambda_3=\pm 1$. 
\end{proof}

As a result, we may rescale $w_0$ so that $\lambda_1=\pm 1$, $\lambda_3=\pm 1$. 

\vspace{.5\baselineskip}

Now we study the band surgery map on $C$ when $C$ has an arbitrary number of components. 

\begin{notn}
	\label{notn_ordering_of_components}
Fix an embedding of $\Sigma$ into $\bR^2$. 
For each $1$--manifold $C$ embedded in $\Sigma$, choose an ordering of the components of $C$ as follows.  
Suppose $\ga_1$ and $\ga_2$ are parallel non-contractible circles in $C$. Since we embed $\Sigma$ into $\bR^2$, one of $\{\ga_1,\ga_2\}$ is inside the other. Assume $\ga_1$ is on the inside and $\ga_2$ is on the outside. Then we require that $\ga_1$ appears before $\ga_2$ in the ordering. 
When there are more than one  orderings satisfying this condition, we choose one arbitrarily and fix the choice. Let $\sigma$ be the chosen ordering of the components of $C$ and  let $\sigma_b$ be the ordering for $C_b$.
\end{notn}

Recall that the isomorphism $\Theta_{w_0,\sigma_b}:V_\bC(C_b)\to \SiHI(C_b)$ is only defined up to changes of signs on the images of the standard basis. We make an arbitrary choice of signs and use it to define the inverse map $\Theta_{w_0,\sigma_b}^{-1}$. Define 
\begin{equation}
	\label{eqn_defn_T(b)}
T(b) =\Theta_{w_0,\sigma_b}^{-1} \circ \SiHI(b)\circ  \Theta_{w_0,\sigma}: V_\bC(C) \to V_\bC(C_b).
\end{equation}

\begin{prop}
	\label{prop_description_of_T(b)}
	Write $C = C'\cup C_0$, $C_b = C_b'\cup C_0$, where $C'$ and $C_b'$ consist of the components of $C$ and $C_b$ that intersect $b$, and $C_0$ consists of the components that are disjoint from $b$. Then
	$T(b)$ is given by the tensor product of a homomorphism from $V_\bC(C')$ to $V_\bC(C_b')$ and the identity map on $V_{\bC}(C_0)$, where the homomorphism from $V_\bC(C')$ to $V_\bC(C_b')$ is described by Lemmas \ref{lem_cobordism_SiHI_one_trivial_to_two_trivial} to \ref{lem_SiHI_cobordism_two_nontrivial_one_nontrivial} up to multiplying integer powers if $i$ to the coefficients of the equations. 
\end{prop}

\begin{proof}
	By Lemma \ref{lem_change_order_with_contractible}, we can change the position of any contractible circle in $\sigma$ or $\sigma_b$ without affecting the statement of the proposition.
	
	Note that if $\ga_1$ and $\ga_2$ are a pair of disjoint, non-parallel, non-contractible circles on $\Sigma$, then the condition of Lemma \ref{lem_change_order_non_parallel} is satisfied.
	Hence by Lemma \ref{lem_change_order_non_parallel} and the associativity of $\clP$, we may switch the order of a pair of consecutive non-parallel, non-contractible components in $\sigma$ or $\sigma_b$ without affecting the statement of the proposition. 
	
	Therefore, we may assume without loss of generality that the components of $C'$ and $C_b'$ are placed in consecutive orders in $\sigma$ and $\sigma_b$ respectively.  The desired result then follows from Lemmas \ref{lem_cobordism_SiHI_one_trivial_to_two_trivial} to \ref{lem_SiHI_cobordism_two_nontrivial_one_nontrivial} and  the associativity of $\clP$.
\end{proof}

Now we prove a rank estimate that is central to the proof of Theorem \ref{thm_main}.

\begin{prop}
	\label{prop_comparison_rank_APS_SiHI}
	Suppose $L$ is a link in the interior of $[-1,1]\times \Sigma$. Then
	$$\rank_{\bZ/2} {\rm APS(L)}\ge \dim_\bC\SiHI_{R,p}(L).$$
\end{prop}

\begin{proof}
Suppose the link $L$ is given by a diagram $D$ on $\Sigma$ with $k$ crossings. Fix an ordering of the crossings.
For $v=(v_1,v_2,\dots, v_k)\in \{0,1\}^k$, let $D_v$ be the resolved diagram indexed by $v$.

Let $\mathcal{S}\in \{0,1\}^k \times \{0,1\}^k$ be the set of pairs $(v,u)$ such that $u$ is obtained from $v$ by changing one component from $0$ to $1$. For $(v,u)\in \mathcal{S}$, the change from $D_v$ to $D_u$ is given by a band surgery
 (see Figure \ref{fig_01smoothing} and Figure \ref{fig:band_surgery}). We use $b_{vu}$ to denote the band. 
	
By \cite[Proposition 6.9]{KM:Kh-unknot}, there is a spectral sequence abutting to 
$$
\II(S^1\times R, L,S^1\times \{p\})
$$ 
whose first page is 
$$
\bigoplus_{v\in\{0,1\}^k} \II(S^1\times R, \{0\}\times D_v,S^1\times \{p\}).
$$
The differential on the first page is given by the sum of maps
\begin{equation}
	\label{eqn_second_page_differential}
	\II(S^1\times R, \{0\}\times D_v,S^1\times \{p\}) \to \II(S^1\times R, \{0\}\times D_u,S^1\times \{p\})
\end{equation}
for all $(v,u)\in \mathcal{S}$, 
where each map is equal to the cobordism map defined by $b_{vu}$ up to sign.

According to the discussion in \cite{AHI}*{Section 5} (also cf. \cite{xie2020instantons}*{Lemma 5.4}), this yields a spectral sequence relating the eigenspaces of the $\muu$ maps.
Therefore, we obtain a spectral sequence abutting to $\SiHI_{R,p}(L)$ whose first page is given by 
$$
E_1 = \bigoplus_{v\in\{0,1\}^k} \SiHI(D_v),
$$
and the differentials on the first page is equal to
 $$d_1 = \sum_{(v,u)\in \mathcal{S}}\SiHI(b_{vu})$$ 
for some choice of signs of the cobordism maps. The spectral sequence then implies
\begin{equation}
	\label{eqn_rank_estimate_SiHI_APS_1}
\dim_\bC H(E_1,d_1)\ge \dim_\bC \SiHI_{R,p}(L),
\end{equation}
where $H(\cdot)$ denotes the homology of the complex.

Identify $E_1$ with 
$$\bigoplus_{v\in\{0,1\}^k} V_\bC(D_v)$$
 using the isomorphisms 
 $$\Theta_{w_0,\sigma_v}: V_\bC(D_v)\to \SiHI(D_v),$$ 
 where $\sigma_v$ are the orderings of components given by Notation \ref{notn_ordering_of_components}, and $w_0$ is chosen so that Lemma \ref{lem_cobordism_map_1_nontrivial_to_2_nontrivial} holds. By Lemma \ref{lem_lambda1_lambda3_product}, we may also rescale $w_0$ so that $\lambda_1=\pm 1$ and $\lambda_3=\pm 1$ in Lemmas \ref{lem_SiHI_cobordism_one_trivial_two_nontrivial} and \ref{lem_SiHI_cobordism_two_nontrivial_merge_to_trivial}.

We now define a filtration on $(E_1,d_1)$ following the strategy of  \cite[Section 6.2]{winkeler2021khovanov}.
For each circle $\ga$ on $\Sigma$, define a grading on $V_\bC(\ga)$ as follows. If $\ga$ is a contractible circle, define the grading to be zero. If $\ga$ is a non-contractible circle, recall that we view $\Sigma$ as a subset of $\bR^2$.
 Define the grading of $\bfw_\mfo(\ga)$ to be $1$ if $\mfo$ is the counter-clockwise orientation, and define the grading of $\bfw_\mfo(\ga)$ to be $-1$ if $\mfo$ is the clockwise orientation. 
 
 We adopt the convention that if $B$ is the standard disk in $\bR^2$ then the induced orientation on $\partial B$ is counter-clockwise.
 It is then straightforward to check using Proposition \ref{prop_description_of_T(b)} that the map $d_1$ does not increase the grading, and the only terms in $d_1$ that decrease the grading are the terms with coefficients $\lambda_2$ and $\lambda_4$ given by Lemma \ref{lem_cobordism_map_1_nontrivial_to_2_nontrivial} and Lemma \ref{lem_SiHI_cobordism_two_nontrivial_one_nontrivial}.
 
 Let $(E_1',d_1')$ be the associated graded complex of $(E_1,d_1)$ with respect to the above filtration.
By the spectral sequence associated with filtrations, we have
\begin{equation}
	\label{eqn_rank_estimate_SiHI_APS_2}
\dim_\bC H(E_1',d_1')\ge \dim_\bC H(E_1,d_1).
\end{equation}

The space $E_1'$ is isomorphic to $$\bigoplus_{v\in\{0,1\}^k} V_\bC(D_v),$$
which is the same as the group $\CKh_\Sigma(L)\otimes \bC$ from the definition of $\APS(L;\bC)$ in Section \ref{subsec_APS}. 
By Proposition \ref{prop_description_of_T(b)},   the differential $d_1'$ coincides with the differential on $\CKh_\Sigma(L)\otimes \bC$  up to multiplications by integer powers of $i$ on the images of the standard basis.

This implies that there is a complex 
$$
(\hat{E}, \hat{d})
$$ 
with coefficients in $\bZ[i]$, such that when reducing to $\bC$ coefficients, the complex $(\hat{E}\otimes\bC, \hat{d}\otimes \id_{\bC})$ is isomorphic to $(E_1',d_1')$; when reducing to $\bZ[i]/(1-i) \cong \bZ/2$ coefficients, the complex $$(\hat{E}\otimes \big(\bZ[i]/(1-i)\big), \hat{d}\otimes \id_{\bZ[i]/(1-i)})$$ is isomorphic to the complex that defines $\APS(L;\bZ/2)$.

 By the universal coefficient theorem, we have
 $$
 \rank_{\bZ/2} \APS(L;\bZ/2) \ge \rank_{\bZ[i]} H(\hat{E}, \hat{d}) = \dim_\bC H(E_1',d_1').
 $$
Hence the desired result is proved by combining the above inequality with \eqref{eqn_rank_estimate_SiHI_APS_1}, \eqref{eqn_rank_estimate_SiHI_APS_2}.
\end{proof}

\begin{proof}[Proof of Theorem \ref{thm_main}]
	Suppose $L$ is a link in the interior of $[-1,1]\times \Sigma$ such that $\rank_{\bZ/2}\APS(L;\bZ/2)\le 2$. 
	Let $F,R,p$ be as before.
	Then by Proposition \ref{prop_comparison_rank_APS_SiHI} and \eqref{eqn_SiHI_iso_SHI}, we have
	$$
	\dim_{\bC}\SHI([-1,1]\times \Sigma, \{0\}\times \Sigma, L)\le 2.
	$$
	Now the theorem follows from Theorem \ref{thm_main_instanton}.
	
\end{proof}

\bibliographystyle{amsalpha}
\bibliography{references}

\end{document}